\documentclass{article}

\usepackage{amssymb,amsmath, amsthm}

\usepackage{hyperref}

\newtheorem*{cor}{Corollary}
\newtheorem{prop}{Proposition}

\newtheorem{clm}{Claim}
\newtheorem{defin}{Definition}

\newtheorem{theo}{Theorem}
\newtheorem{coro}{Corollary}
\newtheorem{lemm}{Lemma}
\newtheorem{rema}{Remark}
\newtheorem{defi}{Definition}

 \def\NN{{\mathbb N}}  
 \def\RR{{\mathbb R}}  \def\TT{{\mathbb T}}
   
 \def\ZZ{{\mathbb Z}}

\def\La{\Lambda}

\def\Om{\Omega}
\def\Ga{\Gamma}

\def\cA{{\cal A}}  \def\cG{{\cal G}}  
   \def\cN{{\cal N}} 
\def\cC{{\cal C}}   \def\cO{{\cal O}} 
    
\def\cE{{\cal E}}    
\def\cF{{\cal F}}

\title{Flexible periodic points}

\author{Christian Bonatti \and
Katsutoshi Shinohara}

\begin{document}

\maketitle
\begin{abstract}
 We define the notion of 
$\varepsilon$-flexible periodic  point: 
it is a periodic point with stable index equal to two 
whose dynamics restricted to the stable direction 
admits $\varepsilon$-perturbations 
both to a homothety and a saddle having an eigenvalue equal to one. 
 We show that $\varepsilon$-perturbation to 
an $\varepsilon$-flexible point allows to change it 
in a stable index one periodic point whose 
(one dimensional) stable manifold is an arbitrarily chosen $C^1$-curve. 
 We also show that the existence of flexible point is 
a general phenomenon among systems with a robustly 
non-hyperbolic two dimensional center-stable bundle. 

{ \smallskip

 }

\end{abstract}


\section{Introduction}

Since the Poincar\'{e}'s discovery of the transverse homoclinic intersections
and the complex behaviors near them, the search of the transverse homoclinic
intersections, in other words, the control of the invariant manifolds of the systems,
has been one of the central problems in dynamical systems. 
In the late nineties of last century, a breakthrough was brought by 
Hayashi's connecting lemma (see \cite{H}).
It allows us to control the effect of perturbations on the invariant manifolds 
and enables us to create new intersections. 
This perturbation technique provides one basis 
of recent strong development of 
study of non-uniformly hyperbolic dynamical systems. 
For example,  \cite{CP}  uses connecting lemma and 
its generalizations for building heterodimensional cycles 
in order to characterize the robust non-hyperbolic behaviors. 

In the uniformly hyperbolic context, 
the invariant manifold theorem tells us their rigidity.
The local stable/unstable manifolds of a uniformly hyperbolic 
set are discs varying continuously with respect to the points 
and the variation of the diffeomorphisms. 
A consequence of this fact is the great constraint of the 
geometric behavior of invariant manifolds under small perturbations.
Such rigidities are paradoxically exploited for the construction 
of robustly non-hyperbolic systems such as 
Newhouse's example of robust tangencies (see \cite{N}) or 
Abraham-Smale's example  of 
robust heterodimensional cycles (see \cite{AS}). 

Meanwhile, the control of the variations of 
stable and unstable manifolds of periodic orbits 
under small perturbations in non-uniformly hyperbolic 
systems still remains to be an important issue.
On one hand, non-hyperbolic systems in general contain 
plenty of regions where the local dynamics exhibits  
hyperbolic behaviors, which also give us some rigidity 
of the invariant manifolds.
On the other hand, we may expect that the absence
of uniform hyperbolicity implies
the existence of periodic orbits whose invariant manifolds
has less rigidity so that
it can be altered considerably by small perturbations. 

A prototype of such argument can be found in \cite{BD1}.
The rescaling invariant nature of $C^1$-distance
gives a strong freedom for the change of relative 
position of objects in the homothetic regions,   
that is, contracting or repelling regions where 
the diffeomorphism is smoothly conjugated to a homothety. 
Furthermore, \cite{BD1} gave an example of non-uniformly 
hyperbolic systems in which the homothetic behaviors are 
quite abundant.  This enables us to construct 
interesting examples of dynamical systems,
such as the construction of \emph{universal dynamics} by the first author with D\'{i}az,
or the construction of heterodimensional cycles near the 
wild homoclinic classes by the second author \cite{S}.

The fact that non-hyperbolic dynamics may produce 
homothetic behavior come back to Ma\~n\'e argument in 
\cite{M} for surface diffeomorphism (see also \cite{PS}), and has been 
generalized in higher dimension in \cite{BB, BDP,BGV}.
These works are essentially within 
the scope of perturbations of the derivative
of periodic orbits and the conclusions of them 
provides us local information of perturbed systems.  
In this article, we pursue the possibility of such strategy 
and propose new, semi-local technique of the control of invariant manifolds. 

For many applications, 
we need to control the effect of the perturbation on the invariant manifold, 
for example  to keep an heteroclinic connection under  
the periodic orbit is changing of index. 
This is the aim of \cite{Gou} where Gourmelon uses invariant 
cone fields for keeping the strong stable manifold almost unchanged 
along the perturbation.  Here we follow a completely opposite strategy: 
we will use homothetic region (where there is no 
strictly invariant cone field) for obtaining a great freedom 
of choosing the invariant manifold after perturbing 
the derivative of the periodic orbit.

The aim of this paper is twofold. 
First, we consider diffeomorphisms of two dimensional manifolds 
and define the notions of flexible periodic points, which 
is an abstract sufficient condition on a periodic point such that 
the above strategy is well available.  
We investigate the possible perturbation of such points.
We extend the concept of flexible points
to diffeomorphisms in higher dimension 
stable index two periodic point and see 
that the perturbation technique proved in two dimensional 
setting is valid even in higher dimensional situations.
Second, we show that flexible points are 
quite abundant in some type of higher dimensional 
partially hyperbolic dynamical systems.

Let us now state our main results.

\subsection{Flexible points of surface diffeomorphisms}

First, we briefly review the notion of linear cocycles. 
Let $X$ be a topological space, $f :X \to X$ be 
a homeomorphism of $X$ and $E$ be a Riemannian vector bundle over $X$. 
A \emph{linear cocycle} on $\cE$
is a bundle isomorphism $\mathcal{A} : \cE \to \cE$ which is compatible 
with $f$.  In this article, we are mainly interested in the situation where
$f$ is a diffeomorphism of some manifold, $X$ is a periodic orbit 
of such dynamical systems and $\mathcal{A}$ is the restriction of differential map 
acting on the restriction of tangent bundle over the orbit. 
Such a system can be, by taking coordinate, identified with the situation where
$X = \mathbb{Z} / n \mathbb{Z}$ ($n$ is the period of the orbit),
$f(x) := x +1$ and $\cA$ is a sequence of regular matrices.
We call such system \emph{a linear cocycle over periodic orbit of period $n$}. 

Let $\mathcal{A}$, $\mathcal{B}$ be linear cocycles on $\cE$ over $(X, f)$. We put 
\[
\mathrm{dist}(\mathcal{A}, \mathcal{B}) := 
\sup \| \mathcal{A}(x) - \mathcal{B}(x) \| 
\]
and call it \emph{distance} between 
$\mathcal{A}$ and $\mathcal{B}$ ($|x|$ ranges over all unit vectors in all fibers).
This defines a topology on the space of linear cocycles on $\cE$.
Let $\mathcal{A}_t$ denote a continuous one-parameter family of cocycles, 
that is, a continuous map from some interval to the space of cocycles. 
We put
\[
\mathrm{diam}(\mathcal{A}_t) := 
\sup_{s< t}  \mathrm{dist} (\mathcal{A}_s, \mathcal{A}_t)
\]
and call it \emph{diameter} of $\mathcal{A}_t$

Now, let us give a precise definition of \emph{flexible cocycles}.

\begin{defi}\label{d.flexible}
Let $\cA=\{A_i\}_{i\in\ZZ /n\ZZ }$, $A_i\in \mathrm{GL}(2,\RR)$ 
be a linear cocycle over a periodic orbit of period $n > 0$. Fix $\varepsilon>0$.
We say that $\cA$ is \emph{$\varepsilon$-flexible} if there is a continuous 
one-parameter family of linear cocycles $\cA_t=\{A_{i,t}\}_{i\in\ZZ /n\ZZ}$ 
defined over $t \in [-1, 1]$ such that:
\begin{itemize}
\item $\mathrm{diam}(\mathcal{A}_t)  < \varepsilon$;
\item $A_{i,0}=A_i$, for every $i\in\{0,\dots,n-1\}$;
\item $A_{-1}$ is an homothety;
\item for every $t\in (-1,1)$, the product 
$A_t:= (A_{n-1,t})\cdots (A_{0,t})$ has 
two distinct positive contracting eigenvalues.
\item Let $\lambda_t$ denote the smallest eigenvalue of 
 the product $A_{t}$. Then 
 $\max_{-1 \leq t \leq 1} \lambda_t <1$.
\item $A_1$ has a real positive eingenvalue equal to $1$. 
\end{itemize}
\end{defi}

\begin{defi}
A periodic orbit of a two dimensional 
diffeomorphism is called $\varepsilon$-flexible 
if the linear cocycle of the  derivatives along 
its orbit is $\varepsilon$-flexible.
\end{defi}

The interesting feature of $\varepsilon$-flexible points is that 
it has a great freedom for changing the position of 
their strong stable manifold by an $\varepsilon$-perturbation 
supported in an arbitrarily small neighborhood of the orbit.
More precisely, we can choose the 
fundamental domain of this strong stable manifold 
to be any prescribed curve subject to the unique 
limitation that it should remain in the same isotopy class
in the orbit space.  

Let us explain that. 
Let $\mathcal{N}$ be a compact neighborhood of some attracting 
periodic orbit $\mathcal{O}(x)$ of some diffeomorphism 
$F$ on two dimensional manifold.  
Suppose that:
\begin{itemize}
 \item the derivative of $F$ on $\mathcal{O}(x)$ in the period has 
two distinct contracting 
eigenvalues.
\item $\mathcal{N}$ is $F$-invariant and is contained in the basin of $\mathcal{O}(x)$. 
\end{itemize}

Let us consider the \emph{punctured neighborhood}
$\mathcal{N} \setminus \mathcal{O}(x)$ and take the quotient space 
by the orbit equivalence (that is, $x \sim y$ if and only if $F^m(x) = F^n(y)$
for some $m, n \geq 0$). 
We denote it by $T^{\infty}_F$ and call it \emph{orbit space}.
Then, $T^{\infty}_F$ is naturally 
identified with the two dimensional torus $\TT^2$, 
This torus is naturally endowed with a homotopy 
class of a \emph{parallel}, which consists in the class of an 
essential circle around $\mathcal{O}(x)$.
Moreover, the strong stable manifold $W^{ss}(x)\setminus \{x\}$ 
projects to the quotient space as two parallel circles, 
cutting the parallel with intersection number $1$. 
We call these two simple closed curves in $T_F^\infty$ \emph{meridians}.

Then, consider another diffeomorphisms  $G$  which is a
perturbation of $F$ with small support concentrated around 
$\mathcal{O}(x)$ and preserving the orbit of $\cO(x)$.
Let $\Lambda_G := \cap_{i \geq 0}G^{i}(\mathcal{N})$
denote the locally maximal invariant set in $\mathcal{N}$. 
Now, consider the set in $\mathcal{N} \setminus \Lambda_G$ 
and take the quotient by the orbit equivalence under $G$. 
We denote the orbit space by $T^{\infty}_G$.
We identify $T^{\infty}_F$ and $T^{\infty}_G$ as follows.
First, note that the restriction of $G|_{\mathcal{N}\setminus\Lambda_G}$, 
can be conjugated to $F|_{\mathcal{N}\setminus\mathcal{O}(x)}$ by a unique 
homeomorphism $h$ which coincides with the 
identity map outside the support of the perturbation 
and is $C^1$ out of $\mathcal{O}(x)$. 
Let us call this homeomorphism \emph{standard conjugacy}. 
It gives us an identification between $T^{\infty}_F$ and  $T^{\infty}_G$

Under this identification, the freedom of flexible point
mentioned above can be formulated as follows.
\begin{theo}\label{t.flexible}
Let $f$ be a $C^1$-diffeomorphism of a surface and $D$ be an attracting 
periodic disc of period $\pi$, that is, 
$D$, $f(D),\dots f^{\pi-1}(D)$ are pairwise disjoint 
and $f^\pi(D)$ is contained in the interior of $D$. 
Assume that $D$ is contained in the stable manifold of 
an $\varepsilon$-flexible periodic point $p$ contained in $D$. 
Let $\gamma=\gamma_1 \cup \gamma_2 \subset T^{\infty}_{f}$ 
be the two simple closed curves which $W^{ss}(p)$ projects to.

Then, for any pair of $C^1$-curves $\sigma=\sigma_1\cup \sigma_2$ 
embedded in $T^{\infty}_{f}$ isotopic to $\gamma_1 \cup \gamma_2$,
there is an $\varepsilon$-perturbation $g$, 
supported in an arbitrarily small neighborhood of $p$
(which is also sufficiently small so that 
we can define the standard conjugacy) such that $g$ satisfies 
the following:
\begin{itemize}
 \item $p$ is a periodic attracting point having a eigenvalue 
$\lambda_1\in]0,1[$ and a eigenvalue $\lambda_2=1$. 
\item $D$ is contained in the basin of $p$.
\item $W^{ss}(p, g)$ projects to 
$(\sigma_1 \cup \sigma_2) \subset T^{\infty}_{g} \simeq T^{\infty}_{f}$.
\end{itemize}
\end{theo}

The orbit of $p$ for $g$ is a non-hyperbolic attracting point, 
having an eigenvalue equal to $1$. 
By an extra, arbitrarily small perturbation, 
one can change the index of $p$ so that the strong stable 
manifold becomes the new stable manifold. Therefore we have the following.
\begin{cor} Under the hypotheses of Theorem~\ref{t.flexible},  
there is an $\varepsilon$-perturbation $h$ of $f$, supported 
in an arbitrarily small neighborhood of $p$
so that 
\begin{itemize}
 \item $p$ is a periodic sadlle point having two real eigenvalues $0<\lambda_1<1<\lambda_2$.
\item $W^s(p)\setminus \cO(p)$ is disjoint from the maximal invariant set $\Lambda_H$.
\item $W^{s}(p, h)$ projects to 
$(\sigma_1 \cup \sigma_2) \subset T^{\infty}_{h}\simeq T^{\infty}_f$.
\end{itemize}
 \end{cor}

The proof of Theorem \ref{t.flexible}
is given in Section 2, 3 and 4.

\subsection{Stable index two flexible points}

The flexible points are certainly not usual for surface diffeomorphisms.  
However, it appears very naturally in 
higher dimensional partially hyperbolic systems with two dimensional stable directions.
To explain it, first we extend the definition of flexible points in higher dimensional setting
and see the direct consequence of Theorem \ref{t.flexible} about them.
We prepare one notion. Let $f$ be a diffemorphism of some differential manifold, 
and $p$ be a (not necessarily hyperbolic) periodic point of it. 
Then the \emph{stable index} of $p$, is the number 
of eigenvalues of the differential of the first return map with absolute value 
strictly less than one counted with multiplicity. 

\begin{defin}Let $M$ be a compact manifold endowed with a Riemann metric, 
$f$ a diffeomorphism of $M$ and $\varepsilon>0$.  
A stable index $2$ hyperbolic period point $x$ is called $\varepsilon$-flexible 
if the restriction of $Df$ to the stable direction
over $\mathcal{O}(x)$ is an $\varepsilon$-flexible cocycle. 
\end{defin}

\begin{rema}\label{r.rflex}
 The notion of flexibility is a robust property in the 
following sense: if $q_f$ is $\varepsilon$-flexible periodic orbit of $f$
 then there is a $C^1$-neighborhood $\mathcal{U}$ of $f$ so that  every $g\in\mathcal{U}$ has a well defined continuation $q_g$ of $q_f$ and 
 $q_g$ is $2\varepsilon$-flexible. 
\end{rema}

Let $x$ be a stable index $2$ periodic point of period $n$, with 
two distinct real positive contracting eigenvalues.
For flexible points in this setting, we can also define the notion of
orbit space, parallel, and meridians as follows.

Consider the neighborhood $\cN$ in the local 
stable manifold of $x$ which is strictly positively invariant.  
The orbit space of $\cN \setminus \mathcal{O}(x)$ under $f$
is again diffeomorphic to the torus $\TT^2$ and we denote it by $T_f^{\infty}$.  
As is in the previous case, this torus is naturally endowed with parallel and 
the strong stable manifold of $x$ induce by projection 
on $T^{\infty}_f$ two disjoint simple curved called meridians.  

We consider perturbation $g$ of $f$ whose support 
is concentrated to $\cO(x)$ and which 
preserves the forward invariant property of $\mathcal{N}$.
Then,
the space of $g$-orbits in $D \setminus (\Lambda_g)$
(where $\Lambda_g|$ is the locally maximal invariant set 
of $g$ in $\cN$), denoted by $T^{\infty}_g$ are naturally identified with
$T^{\infty}_f$ via the standard conjugacy.

It is easy to observe the following:

\begin{rema} Any $\varepsilon$-perturbation of the restriction 
of $f$ to the positive orbit of $\cN$, 
in a sufficiently small neighborhood of $\cO(x)$ can be realized 
as a $\varepsilon$-perturbation of $f$.
Furthermore, if the perturbation prerserves the periodic 
orbit $\cN$ then one may required that the eigenvalues of 
$x$ transverse to $\cN$ are kept unchanged.  
\end{rema}

Therefore, as a direct corollary of Theorem~\ref{t.flexible} we obtain the following:
\begin{coro} \label{c.fleper}
Let $f$ be a diffeomorphism of a compact manifold, 
$\varepsilon>0$, and $x$ be a 
$\varepsilon$-flexible stable  index $2$ hyperbolic periodic point. 
Fix a strictly forward invariant neighborhood $\cN$ of $\mathcal{O}(x)$ 
contained in the 
local stable manifold of $x$. Let $T^{\infty}_f$, endowed with the meridians 
$\gamma_1,\gamma_2$, denote the orbit space of the 
punctured stable manifold of $x$. 

Let $\sigma_1$ and $\sigma_2$ be two disjoint simple curves 
isotopic to the meridians. then there is an 
$\varepsilon$-perturbation $g$ of $f$, supported in an arbitrarily 
small neighborhood of $\cO(x)$, preserving the forward invariance of $\cN$, 
preserving the orbit $\cO(x)$ such that the following holds:
\begin{itemize}
 \item The eigenvalues in the directions 
 transversal to $\cN$ at $x$ of $g$ are same for that of $f$.
 \item $\cO(x)$ is a stable index $1$ sadlle point: there is a contracting eigenvalue tangent to $\cN$. 
 \item the punctured stable manifold of $x$ is disjoint form the maximal invariant in $\cN$, and the projection of 
 $W^s(x)\setminus\{x\}$ on $T^{\infty}_f$ is precisely the two curves $\sigma_1\cup\sigma_2$.
\end{itemize}
\end{coro}

As a conclusion, for flexible points in this setting,
we also have great freedom for the choice of position of 
strong stable manifolds.
Meanwhile, there is a difference.
In higher dimensional setting, we have a priori few control 
on the effect of the perturbation to the local unstable manifold of $x$, 
and therefore on the position of the local strong unstable manifold of $g$.  

More precisely, if the angles between the unstable bundle 
over $\cO(x)$ and the orthogonal to the stable bundle
is very small all along $\cO(x)$, then every $\varepsilon$ 
perturbation of $f$ in the local stable manifold of $\cO(x)$,
supported in a very small neighborhood of $\cO(x)$ and 
preserving $\cO(x)$ may be realized as a 
$\varepsilon$-perturbation $g$ of $f$, 
supported in a small neighborhood of $\cO(x)$ and which 
coincides with $f$ on the local unstable manifold of $\cO(x)$. 
On the other hand, this angle is large, that is, if the angle between the stable 
and unstable bundles is very small at some point of 
$\cO(x)$, then perturbing $f$ in the stable manifold without 
perturbing the unstable one can costs a lot.

Such a difference can be a serious problem if one applies
this technique to the problem of (non-)existence of homoclinic intersections.
However, if we have a priori estimates on the angles mentioned above, 
the perturbation technique suggested in Corollary \ref{c.fleper} works well. 
Note that such a priori  estimates are available in the case where the system
admits partially hyperbolic splitting, or more generally, dominated splittings.

\subsection{Abundance of flexible points}

At a first glance, the definition of flexible points may look strange, 
as it claims the existence of perturbations to two completely different situations. 
However, the fact is that 
such a situation is quite abundant in the context of 
non-uniformly hyperbolic situations as we will see below. 

First let us give one precise statement in the form of a $C^1$-generic property.
\begin{theo}\label{t.generic} 
There is a residual subset $\cG\subset \mathrm{Diff}^1(M)$ so that 
for every $f \in\cG$, for any 
chain recurrence class $C$ containing
\begin{itemize}
 \item a periodic point $p$ of stable index $2$ with complex (non real) contracting eigenvalue
 \item a periodic point $q$ of stable index $1$
\end{itemize}
and for any $\varepsilon>0$ there are $\varepsilon$-periodic 
point $p_n$ homoclinically related to $p$ and whose orbits 
$\gamma_n$ converges to  the chain recurrence class $C$ in the 
Hausdorff topology.
\end{theo}

Let us see that there are large class of diffeomorphisms satisfying
the hypotheses of Theorem \ref{t.generic}.
To explain it, let us briefly review the notion of robust heterodimensional cycles.
We say that two hyperbolic basic sets $K$ and $L$ of 
a diffeomorphisms $f$ form a 
\emph{$C^1$-robust heterodimensional cycle} if 
\begin{itemize}
 \item the stable-indices of  $K$ and $L$ are different,
 \item for any $g$ sufficiently $C^1$-close to $f$ the continuations $K_g$and $L_g$ of
 $K$ and $L$ satisfies that 
 $$W^u(K_g)\cap W^s(L_g)\neq \emptyset\mbox{ and }W^s(K_g)\cap W^g(L_g)\neq \emptyset.$$
\end{itemize}
Robust heterodimensioal cycles are 
very important mechanisms for the study of robustly non-hyperbolic behaviors
of diffeomorphisms, as they are the mechanisms  that 
account for the birth of robust non-hyperbolicity
in the large class of $C^1$ non-hyperbolic diffeomorphisms.
Indeed, up to now, all the known examples of robustly non-hyperbolic 
behaviors are ascribed to robust heterodimensional cycles, 
and it is conjectured by the first author in \cite{B} 
that every robustly non-hyperbolic diffeomorphisms
could be approximated by one which has 
a robust heterodimensional cycles.
Furthermore, 
it is worth mentioning that the creation of robust heterodimensional cycle 
is a quite general phenomenon from the bifurcation of heterodimensional cycles 
between saddles of different indices (see \cite{BD2} and \cite{BDK}).

Then, let us consider diffeomorphisms satisfying following conditions:
\begin{itemize}
\item It has a robust heterodimensional cycle between two hyperbolic 
basic sets $K$ and $L$.
\item $K$ has stable index one and $L$ has stable index two.
\item $L$ has a periodic point with complex eigenvalues to the stable directions.
\end{itemize}
The set of such diffeomorphisms forms, by definition, an open set in 
$\mathrm{Diff}^{1}(M)$ and non-empty if $\dim M \geq 3$.
Then, every diffeomorphism contained in the intersection of this 
open set and the residual set in Theorem \ref{t.generic} serves as 
an example (for the chain recurrence set take the one that contains $K$ and $L$).

To observe the largeness of the class of diffeomorphisms which 
are within the range of hypotheses of Theorem \ref{t.flexible}, 
let us briefly discuss the relationship between the hypotheses of Theorem \ref{t.flexible} 
and the notion of \emph{homoclinic tangencies}.
The hypotheses of Theorem \ref{t.flexible} requires the existence 
of periodic point which has complex eigenvalues to the stable direction.
This implies the indecomposability of the stable directions. 
The other condition guarantees that the stable direction is not uniformly 
contracting. By the work of Gourmelon \cite{Goutan} and Wen \cite{W},
these condition is equivalent to that this diffeomorphism can be 
approximated by the one with homoclinic tangency of stable index one. 
Remember that, the result \cite{BD2} tells that the class of diffeomorphisms 
which satisfies such conditions are quite large.

The proof of 
Theorem~\ref{t.generic} is a consequence  of Theorem~\ref{t.aboundance} below
combined with the generic property that $C^1$-generically a homoclinic class
coincides with the chain recurrence class which contains it (see \cite{BC})
and the coexistence of the periodic point of different indices implies the 
existence of robust heterodimensional cycles \cite{BD2}. 
To state Theorem~\ref{t.aboundance}, we prepare some definitions.
Given a periodic point $p$ and a neighborhood $U$ of $p$, 
the \emph{relative homoclinic class} 
$H(p,U)$ of $p$ in $U$ is the closure of the set of transverse 
homoclinic points whose whole orbit is  contained in $U$. 
A periodic point $q$ is \emph{homoclinically related with $p$ in $U$} if there are 
transverse intersection orbits which is contained in $U$ 
and going from $p$ to $q$ and from $q$ to $p$.

Now, Theorem~\ref{t.generic} is a consequence of the following ''local" result:

\begin{theo}\label{t.aboundance} Let $f$ be a diffeomorphisms 
of a compact manifold. 
Suppose that $f$ admits a hyperbolic periodic point $p$
and an open neighborhood $U$ of the orbit  $\cO(p)$  with the following properties: 
\begin{itemize}
\item $p$ has stable index equal to $2$. 
 \item there is periodic point $p_1$ homoclinically related with $p$ in $U$, 
 so that the $p_1$ has a complex (non-real) contracting eigenvalues.
 \item there is a periodic point $q$ with $\cO(q)\subset U$ with stable index $1$.
 \item there are hyperbolic transitive  basic sets $K\subset H(p,U)$ and $L\subset H(q,U)$ 
 containing $p$ and $q$, respectively, so that $K,L$ form a $C^1$-robust 
 heterodimensioanl cycle in $U$. 
\end{itemize}

Then, for any $\varepsilon>0$ there are arbitrarily small $C^1$-perturbations $g$ of $f$ having a 
$\varepsilon$-flexible point homoclinically related with $p_g$ in $U$ and which orbit is $\varepsilon$-dense
in the relative homoclinic class $H(p_g,U)$. 
\end{theo}
  
We give the proof of Theorem \ref{t.aboundance}
(thus also Theorem \ref{t.generic}) in Section \ref{s.abunda}.

\subsection{Possible dynamical consequences and generalizations}

The notion of flexible points and its abundance
are interesting in themselves.
At the same time, we think that it would be a powerful tool for the 
study of $C^1$-generic systems in many ways. 
Let us explain some idea with some details. 

The first possible application is
for the investigation of
\emph{tame}/\emph{wild} properties in $\mathrm{Diff}^1(M)$ 
(see \cite{B}). 
In a further work, the authors will use them as a 
mechanism for producing new example of wild diffeomorphisms, that is, 
$C^1$-generic diffeomorphsims with infinitely many chain recurrence classes. 
The idea is very simple: if $p$ is a flexible point of stable index $2$, 
one can transform $p$ to be a stable index $1$ periodic point whose 
stable manifold is an 
arbitrarily chosen curve in the old $2$-dimensional stable manifold.  
If we can choose this curve being disjoint from the initial 
chain recurrence class, then it implies that the point has been ejected 
from the class.  Repeating this procedure, we obtain 
infinitely many saddles with trivial homoclinic classes in 
a neighborhood of any classes satisfying hypothesis of Theorem~\ref{t.generic}. 

However, this strategy is not complete:
there are $C^1$-robustly transitive diffeomorphisms 
which satisfies the hypothesis of Theorem~\ref{t.generic} 
(see for example \cite{BV}).  
They have plenty of flexible points, 
but cannot be expelled to outside the class! 
The reason is that, since the class being the whole manifold, 
there is no space to escape from the original classes.
Thus, to execute this strategy completely, we need to
investigate the \emph{topology of the chain recurrence class}
looked inside the center stable directions, 
which will be one of the central topic in \cite{BS}.

We suggest another possible direction of application.
The control of the position of the stable manifold may open the 
way to study the difference between $C^1$-generic diffeomorphisms 
and the $C^1$-open diffeomorphisms. More precisely, for example,  
it is known that for $C^1$-generic robustly transitive diffeomorphisms, 
the homoclinic class of every hyperbolic periodic points coincides with 
the whole manifold (this is a consequence of Hayashi's connecting 
lemma). 
One interesting question is to see if this property is an open property. 
A priori, there is no reason that it is. Meanwhile, 
to establish a rigorous conclusion is not such a simple matter. 
For that, what we need to understand is the position of homoclinic intersections. 
Under the situation where we have abundance of flexible points,
obtaining the better understanding of the position of homoclinic intersections
sounds quite feasible,
since the perturbation technique 
of Theorem \ref{t.flexible} provides us with a 
strong freedom for the control of the (un)stable manifolds.

In this paper, we define the notion of flexible points of stable index $2$ .  
We can define similar notions for higher stable f cases, for example, 
as points whose derivative to stable direction can be perturbed 
both to a contracting homothety 
and to a saddle having at least one eigenvalue equal to $1$.  
It would be interesting to establish similar perturbation techniques
to choose control the position of stable manifolds for them, 
and  to study possible topology of flags of 
strong stable manifolds in the orbit spaces.
However, creating a deformation of a linear cocycle in higher dimension 
is much more difficult and technical than in dimension $2$.
Therefore, we think it is better to restrict our attention 
to two dimensional situation, and 
leave the generalization to higher dimensional cases 
by the time when we will be well prepared for such ambitious researches. 

\quad

{\bf Acknowledgements} \quad 
During preparing the manuscript, 
the authors were supported ANR 08 BLAN-0264-01 DynNonHyp (postdoctoral grant) and CNPq (Brazil).
The authors thank Sylvain Crovisier, Nicolas Gourmelon, 
Rafael Potrie and Dawei Yang for having discussions.


\section{The flexibility and the control of the stable manifold}

The aim of section 2, 3 and 4 is the proof of Theorem~\ref{t.flexible}.
 It consists of noticing that the flexibility property 
which allows us to perform an $\varepsilon$-perturbation of 
the flexible linear cocycle, among the cocycle 
 of diffeomorphisms, which inserts a region 
where the dynamics in the period is a homothety.  
Furthermore, the number of fundamental domains 
in the homothethic region may be chosen arbitrarily. 
 In some sense, we ask the orbit to lose an arbitrarily large 
time in that homothetic region. 
 For this reason we call them \emph{retardable cocycle}. 
 
 Iterating a homothetic region does not introduce any distortion: 
as we will see, 
 it is therefore easy to control the effect of perturbations 
performed inside the homothetic region. 
 The fact that we may use an arbitrary number of 
fundamental domains leaves 
 us the time to deform slowly the strong unstable manifold 
to the a priori chosen curve.

 \subsection{Retardable cocycle} 
 
To state the meaning of ``inserting a lot of homothetic regions"
we first define the notion of retardable cocycles. 
Let us give the notion of diffeomorphism cocycles over a finite orbit.
We consider cocycles of diffeomorphisms on $\RR^2\times\ZZ/n\ZZ$, 
for which $\{ {\bf 0}_i \}$ 
(where we put ${\bf 0}_i := (0,i)$) is a periodic sink 
attracting all the points. More precisely,  
a \emph{diffeomorphism cocycle} is a set of diffeomorphisms
$\mathcal{F} =\{f_i \mid \mathbb{R}^2 \times 
\{ i\} \to \mathbb{R}^2 \times \{ i+1\} \}$.
We denote the map induced on the total space $\RR^2\times\ZZ/n\ZZ$
from the cocycle $\mathcal{F} =\{f_i\}$ also by $\mathcal{F}$.

In this article, we assume that every diffeomorphism cocycle fixes 
the origin, that is, we always assume
$f_i({\bf 0}_i) ={\bf 0}_{i+1}$ for every $i$.
Such cocycle is called \emph{contracting} if the $0$-section is 
an attracting periodic orbit and if any orbits is contained in its basin. 
Given a linear cocycle $\mathcal{A} = \{A_i\}$, we consider it 
as a diffeomorphism cocycle in the obvious way,
that is, we consider $A_i$ to be the diffeomorphism 
from $\mathbb{R}^2 \times \{ i\}$ to $\mathbb{R}^2 \times \{ i+1\}$.
For a diffeomorphism cocycle $\mathcal{F} = \{f_i\}$, 
we denote the first return map of it on $\mathbb{R}^2$ 
by $F$ (drop the subscript in $\ZZ/n\ZZ$ and capitalize the 
symbol.) Note that a linear cocycle $\mathcal{A} = \{A_i\}$
is contracting if and only if all eigenvalues of $A$ have
absolute value strictly less than one.

In the following, we denote
the two dimensional disk of radius $r$
centered at $\boldsymbol{0}_i$
by $B_i(r)\subset \RR^2\times\{i\}$
and for any $0<r<s$ we denote 
by $\Ga_{r,s}$ the round closed annulus in 
$\mathbb{R}^2 \times \{ 0\}$
bounded by the circles of 
radii $r$ and $s$, that is, $\Ga_{r,s} := 
\overline{B_0(s) \setminus B_0(r)}$.

\begin{defi}
A contracting cocycle of diffeomorphisms  
$\mathcal{F} =\{f_{i} \}$ 
over $\RR^2\times\ZZ/n\ZZ$
is called \emph{retardable cocycle} 
if it satisfies the following conditions:

There exists $R_1, R_2, R_3$ satisfying
$0< R_1 < R_2 <R_3 $ such that 
\begin{itemize}
\item  $f_i|_{B_i(R_1) \setminus B_i(R_3)} = A_i$, where $A_i$
is a linear map such that 
$A =\prod_{j =0}^{n-1} A_j = \lambda \mathrm{Id}$ where
$0 < \lambda <1$. In other words, $A$ is a contracting homothety.
\item For every $x\in B_0(R_2)$ and $i$ satisfying $0 \leq i < \pi$,
we have  $(\prod_{j=0}^{i-1} f_j)(x) \in B_{i}(R_1)$.
\item $A(B_0(R_2))$ contains $B_0(R_3)$ in its interior.
We call $\Gamma_{\lambda R_2 , R_2}$  
\emph{homothetic region} of $\mathcal{F}$.
\end{itemize}
\end{defi}

The main property of retardable cocycles is that one may insert arbitrarily many fundamental domains 
of homothety as follows: 

\begin{prop}

Let $\mathcal{F} =\{f_i\}_{i\in\ZZ/n\ZZ}$ be a retardable cocycle. 
We define new cocycles 
$\mathcal{F}_m =
\{f_{i, m} \mid \RR^2\times\{i\}\to \RR^2\times\{i+1\} \}$, 
$(m \geq 0)$ as follows:
\begin{itemize}
\item For $x \in \RR^2\times\{i\} $ with $|x| >R_3 $, 
$f_{i, m}(x) := f_{i}(x)$.
\item For $x \in \RR^2\times\{i\} $ with $\lambda^{m} R_3< |x| < R_3$, 
$f_{i, m}(x) := A_i(x)$.
\item For $x \in \RR^2\times\{i\} $ with $|x| \leq \lambda^{m} R_3$, 
$f_{i, m}(x) :=  (A^{m} \circ f_{i} \circ A^{-m})(x)$.
\end{itemize}

Then those maps define a $C^1$-diffeomorphisms contracting  cocycle.
We call $\{f_{i, m}\}$ the \emph{$m$-retard} of $\{ f_i \}$.
\end{prop}

The proof of above proposition is obvious, so we omit it.

Roughly speaking, $\{f_{i, m}\}$ is a cocycle that 
is obtained by \emph{insertion of $m$-homothetic fundamental domains} to $\{f_i\}$.
This insertion does not change the main dynamical 
properties of the cocyle: 
the orbits will just spend  
more time in the newly added homothetic region.
More precisely, on the homothetic region, 
the relative position of 
objects such as the strong stable manifold 
are kept intact under the iteration.
For example, we have the following properties of retarded cocycles:
\begin{rema}\label{r.retarded}
\begin{itemize}
\item $\{f_{i, 0}\} = \{f_i\}$.

\item Let $\Gamma_{\lambda R, R}$ be the homothetic 
region of $\mathcal{F}$. Then, 
$F|_{\Gamma_{\lambda^{m+1}R, R}}$ is a homothety 
of rate of contraction $\lambda$.
We call $\Gamma_{\lambda^{m+1}R, R}$
\emph{homothetic region} of $\mathcal{F}_m$.

\item Suppose the origin $\{ {\bf 0}_i \} $ 
has strong stable manifold
$W^{ss}(\boldsymbol{0}_0, \mathcal{F})$ of $\{ f_i\}$. 
Then, we have 
\[
W^{ss}(\boldsymbol{0}_0, \mathcal{F}_m) \cap \Gamma_{\lambda^{l+1}R, \lambda^{l}R} 
= A^{l}(W^{ss}(\boldsymbol{0}_0, \mathcal{F}) \cap 
\Gamma_{\lambda R, R} )
\]
for $l$ satisfying $0 \leq l \leq m$.

\item The item above can be stated more sophisticated way 
by using the language of orbit spaces.
Note that for every $m \geq 0$, 
$\mathcal{F}_m = \{ f_{i, m}\}$ coincides with $ \mathcal{F} = \{f_i\}$
outside some compact neighborhood of the origin. Thus 
the standard conjugacy gives the natural identification 
between $T^{\infty}_{\mathcal{F}_m}$ and $T^{\infty}_{\mathcal{F}_m}$.
Then, the above item is paraphrased that 
$W^{ss}(\boldsymbol{0}_0, \mathcal{F}_m)$ and 
$W^{ss}(\boldsymbol{0}_0, \mathcal{F})$ project to the 
same curve in $T^{\infty}_{\mathcal{F}_m}= T^{\infty}_{\mathcal{F}}$.

\end{itemize}
\end{rema}

Furthermore, under some special circumstance, 
the operation of retarding does not change the dynamics so much. 
To explain that, we introduce the notion of distance 
on diffeomorphism cocycles. 
Let $\{ f_i \}$, $\{g_i\}$ be two diffeomorphism cocyle 
on $\mathbb{R}^2 \times \mathbb{Z} / n \mathbb{Z}$.
We say that $\{ g_i \}$ is a \emph{perturbation of $\{ f_i \}$}
if the support, that is, 
the set $\{ x \in \mathbb{R}^2 \times \{ i \}  \mid f_i(x) \neq g_i(x) \}$
is relatively compact for all $i$.
Suppose  $\{g_i\}$ is a perturbation of $\{ f_i \}$ and $\varepsilon > 0$. 
Then $\{g_i\}$ is called $\varepsilon$-perturbation of $\{f_i\}$
if 
\[
\max_{i \in  \mathbb{Z} / n \mathbb{Z},  x \in \mathbb{R}^2 \times \{ i \}} \| Df_i(x) - Dg_i(x) \|  < \varepsilon.
\]
For the notion of $\varepsilon$-perturbation of 
diffeomorphism cocycle, see Remark \ref{r.ondist} below.
Then, for the retarded cocycle, we have the following.
\begin{lemm}\label{l.retper}
Suppose $\{ A_i \}$ is a linear
cocycle and $\{ f_i \}$ is a retardable diffeomorphism cocycle
which is also an $\varepsilon$-perturbation
of $\{A_i\}$ such that the support of the perturbation 
and the homothetic region is contained 
in a neighborhood $\mathcal{N}$ of $\{ {\bf 0}_i \} $.
Then $\{f_{i, m}\}$ is also a $\varepsilon$-perturbation
of $\{A_i\}$ whose support is contained in $\mathcal{N}$.
\end{lemm}

\begin{rema}\label{r.ondist}
The notion of $\varepsilon$-perturbation gives a concept 
of closeness between a cocycle and its perturbation. 
A priori, this is different from the notion of usual $C^1$-distance,
since our one does not take the contribution of $C^0$-distance 
into consideration.  
However, this difference is negligible for the following reason:
In the following, we establish a perturbation 
technique which provides us an $\varepsilon$-perturbation with 
very small (indeed, arbitrarily small) support. 
This smallness of support combined with the smallness 
of the $\varepsilon$ implies the smallness of $C^{0}$-distance
and supplements the lack of it in the practical use.
\end{rema}

By inserting a lot of homothetic regions, combining the fragmentation
lemma, we will see that we can obtain a considerably 
large freedom to change the relative position of the objects.  

\subsection{Proof of Theorem~\ref{t.flexible}}

The aim of this section is to prove Theorem~\ref{t.flexible} as a  consequence of the 
following propositions.

The first proposition makes the relation between flexible and retardable cocycles: retardable 
cycles may be obtained as small perturbations of flexible cocyles.

\begin{prop}\label{p.flexible-retardable} 
Let  $\varepsilon>0$ and 
$\cA=\{A_i\}_{i\in\ZZ /n\ZZ }$, $A_i\in \mathrm{GL}(2,\RR)$ 
be an $\varepsilon$-flexible linear cocycle over a periodic orbit 
of period $n >0$. 

Then there is a contracting retardable diffeomorphisms cocycle $\mathcal{F}=\{f_i\}$ with the following properties:
\begin{itemize}
 \item For any $m\in \NN$, the retarded cocycle 
$\mathcal{F}_m=\{f_{m,i}\}$ is an $\varepsilon$-perturbation of $\cA$.
 \item There is an isotopy of contracting diffeomorphism cocycle
 connecting $\mathcal{F}$ and $\cA$ 
such that for every moment
the periodic orbit $\{ {\bf 0}_i\}$ has two 
different real eigenvalues.
 \item  For every $ i\in\ZZ /n\ZZ$, the map
$f_i$ coincides with $A_i$ out of  the unit balls 
$B_i(1) \subset \RR^2\times \{i\}$. 
 \item The derivative $DF$ at the origin
$\{ {\bf 0}_0 \}$ has  a contracting eigenvalue 
and one eigenvalue equal to $1$.  
\end{itemize}
\end{prop}

The second proposition explores the effect of perturbations of retardable cocycles 
on the position of the strong stable manifold. 

%

\begin{prop}\label{p.retardable-manifold}
Let $\mathcal{F} = \{ f_i \}$ be a retardable 
diffeomorphism cocycle
over $\mathbb{R}^2 \times \ZZ/n\ZZ$
whose origin has two distinct real 
positives eigenvalues. 
For any pair of disjoint simple curves $\sigma_1,\sigma_2$  
in $T_{\cF}^\infty$, isotopic to the meridians, 
and for any $\varepsilon_0>0$, there is $N>0$ so that, 
for every $m\geq N$ 
there is a $\varepsilon_0$-perturbation 
$\mathcal{G}$ of the $m$-retarded cocycle $\mathcal{F}_m$
such that the following holds:
\begin{itemize}
 \item $\mathcal{G}$ is a perturbation of $\mathcal{F}_m$ with support in the 
homothetic region.
\item  $\mathcal{G}$ is a contracting cocycle. 
\item The strong stable manifold of the 
origin $\{ {\bf 0}_i \}$
induces $\sigma_1\cup \sigma_2$ on the orbit space 
$T_{\mathcal{G}}^\infty$.
\end{itemize}
\end{prop}

Let us give the proof assuming these two perturbation results.

\begin{proof}[Proof of Theorem~\ref{t.flexible} using Propositions~\ref{p.flexible-retardable}
and~\ref{p.retardable-manifold}]

Let $f$ be a $C^1$-diffeomorphism of some surface 
and $D$ be an attracting periodic disc, in the basin of an
$\varepsilon$-flexible  hyperbolic  periodic point $p$ of period $n$. 
Remember that the orbit space in the punctured stable manifold of $p$ is 
a torus $T^{\infty}_f$ endowed with a parallel 
and a meridian (isotopy class of the projection 
of the strong stable manifold of $p$. 

First, we perform a perturbation along the orbit of $p$
so that we can reduce the problem to the linear cocycle case, 
which enables us to use Proposition ~\ref{p.flexible-retardable}
and~\ref{p.retardable-manifold}.
In the following, all the perturbations we give 
are tacitly assumed to be 
supported in sufficiently small neighborhood of the orbit of $p$ 
so that the orbits entering in $D$ can always be 
identified with a point to $T^{\infty}_{f}$ by standard conjugacy. 
We fix a pair $\sigma_1,\sigma_2$ of 
disjoint simple curves isotopic to a meridian.  
Then, by an arbitrarily  $C^1$-small perturbation 
of $f$ supported in an arbitrarily small neighborhood of 
the orbit of $p$, one can obtain a diffeorphism $f_0$ whose 
expression in local charts around the orbit of $p$ is linear, 
and coincides with the differential of $f$ along the orbit of $p$. 
Furthermore, $f_0$ is isotopic to $f$ through cocycles
with the same eignevalues along the orbit of $p$, 
so that the continuous dependance of the strong stable 
manifolds implies that the meridian of $f_0$ in $T^{\infty}_f$ 
are isotopic to the meridian of $f$. 

Therefore, by changing $f$ with $f_0$, 
we can assume that $f$ is linear in a 
neighborhood of the orbit of $p$. 
Let $\cA=\{A_i=Df(f^i(p))\}$ be the corresponding linear cocycle. 
As $f$ is linear near the orbit of $p$ and is a contraction, 
the diffeomorphism $f$ is $C^1$-conjugated to $\cA$  by a 
unique diffeomorphism which  is the identity
map in the small neighborhood of the orbit of $p$. 
Thus  the space of orbits of the punctured linear 
cocycle $T^{\infty}_{\mathcal{A}}$ is canonically identified to $T^{\infty}_f$, 
via this conjugacy. 
Therefore, the circles $\sigma_1$ and $\sigma_2$ of $T^{\infty}_f$ 
induce circles $\alpha_1$ and $\alpha_2$ of $T^{\infty}_{\mathcal{A}}$. 
Now, the problem is translated to the perturbation problem 
of linear cocycles $\mathcal{A}$ over 
$\mathbb{R}^2 \times \mathbb{Z} / n \mathbb{Z}$, that is, 
for proving Theorem \ref{t.flexible}, 
we only need to show that there are $\varepsilon$-perturbations of 
$\cA$ in an arbitrary small neighborhood of the orbit of 
$\{ {\bf 0}_i \}$ 
which satisfies conclusion of the 
theorem for the curves for $\alpha_1$ and $\alpha_2$
in $T^{\infty}_{\mathcal{A}}$. 
Let us construct such perturbation by using 
Proposition~\ref{p.flexible-retardable} and 
\ref{p.retardable-manifold}.

Proposition~\ref{p.flexible-retardable} allows us to 
perform an $\varepsilon$-perturbation of $\cA$ in order 
to get a retardable  cocycle $\mathcal{F}$ which coincides 
with $\cA$ outside the unit ball. 
By conjugating this perturbation by an homothety 
(which does not change the $C^1$-size of the perturbation), 
we can  assume that the support of the perturbation 
given by Proposition~\ref{p.flexible-retardable} 
is contained in an arbitrarily small neighborhood of 
the orbit of the origin. 
Remember that the orbit of the origin for $\mathcal{F}$
and hence for $\mathcal{F}_m$ is a non-hyperbolic 
attracting orbit having exactly one eignevalue equal 
to $1$, as announced in Theorem~\ref{t.flexible}. 
Thus we can talk about the meridians in $T_{\mathcal{F}}^{\infty}$.
Since $\mathcal{F}$ is isotopic to $\cA$ 
through contracting cocycle having distinct real eigenvalues
supported in the small ball, 
we have that the meridians of $\mathcal{F}$ in 
$T^{\infty}_{\mathcal{F}}$ are isotopic 
to the meridians of $\cA$. 
Remark \ref{r.retarded} tells us that the same holds for 
the $m$-retarded cocycles  $\mathcal{F}_m$.

Now we apply Proposition~\ref{p.retardable-manifold}: 
for $m$ large enough, 
${\mathcal{F}}_m$ admits an arbitrarily $C^1$-small perturbation
supported in the homothetic region, so that the strong 
stable manifold of the periodic orbit induces the 
circles $\alpha_1$ and $\alpha_2$ on 
the orbit space $T^{\infty}_A=T^{\infty}_{\mathcal{F}}=T^{\infty}_{\mathcal{F}_m}$. 
\end{proof}

It remains to prove Proositions~\ref{p.flexible-retardable} and \ref{p.retardable-manifold}. In Section \ref{s.pret} we prove
Proositions~\ref{p.flexible-retardable} and 
In Section \ref{s.realiza} we prove \ref{p.retardable-manifold}.


\section{Perturbation of retardable cocycles in the homothetic region}\label{s.pret}

In this section we will prove Proposition~\ref{p.retardable-manifold}. The idea, already appeared in 
\cite{BD1} and  
\cite{BCVW}, is to combine two simple ideas: 
\begin{itemize}
 \item The fragmentation lemma,
which asserts that every diffeomorphism 
of a closed manifold isotopic to the identity map 
can be written as a finite product of diffeomorphisms
arbitrarily close to the identity map, 
supported in balls with arbitrarily small balls.
 \item Conjugating a diffeomorphism 
supported in a small ball by a contracting 
 homothety does not change its $C^1$ distance to 
the identity.  Therefore if one considers 
an $\varepsilon$-perturbation of
a retarded cocycle $\mathcal{F}_m$ supported in some 
fundamental domain in the homothetic region, 
 and if we put this perturbation in another 
fundamental domain by conjugating it by a homothety, 
it remains an $\varepsilon$-perturbation.  
\end{itemize}

Note that the second item is one of the main idea of
Franks' Lemma (linearization of local dynamics near 
the periodic point by arbitrarily small perturbation, see \cite{F}), 
which is frequently used 
in the study of $C^1$-generic dynamical systems.

Let us start the proof.
\begin{proof}[Proof of Proposition~\ref{p.retardable-manifold}]
Let $\mathcal{F}=\{f_i\}$ be a retardable contracting cocycle over
$\mathbb{R}^2 \times \ZZ / n \ZZ$, 
and $T^{\infty}_\mathcal{F}$ the space 
of orbit of the punctured cocycle. 
Let $\gamma_1,\gamma_2 \subset T^{\infty}_{\mathcal{F}}$ be the meridian 
induced by the strong stable manifold of $\mathcal{F}$. 
Remember that for any contracting cocycle 
$\mathcal{G}$ which coincides with $\mathcal{F}$, the punctured orbit space $T^{\infty}_{\mathcal{G}}$ is
identified with $T^{\infty}_{\mathcal{F}}$ through standard conjugacy.
In particular, 
according to remark~\ref{r.retarded},  
we have that $W^{ss}({\bf 0}_0, \cF_m)$
(remember that $\cF_m = \{f_{i, m}\}$ is $m$-retarded cocycle
of $\cF$) 
projects to the same meridian in $T^{\infty}_{\cF} = T^{\infty}_{\cF_m}$
for every $m\in \NN$. 

Let $\sigma_1,\sigma_2\subset T^{\infty}_{\cF}$ be 
two disjoint simple curves  
which are isotopic to meridians. 
We take a diffeomorphism 
$\psi \colon T^{\infty}_{\cF} \to T^{\infty}_{\cF}$ which is isotopic to the identity and 
satisfies $\psi(\gamma_i)=\sigma_i$ for $i =1, 2$. 
Fix some $\varepsilon_0>0$.  
For proving Proposition~\ref{p.retardable-manifold} 
one has to show that there is
$N$ so that every ${\cF}_m$ with $m\geq N$ admits 
a $\varepsilon_0$-perturbation 
supported in the homothetic region, and so that the corresponding meridian are the $\sigma_i$.  


We want to perturb $\cF_m$ in the homothetic region 
to realize the behavior of $\varphi$. For that let us 
first consider the relation between the diffeomorphism 
on the orbit space and that of original space.
Consider a diffeomorphism
$\varphi\colon \RR^2\times\{0\}\to\RR^2\times\{0\}$ 
whose support 
is contained in a fundamental domain of the return map 
$F_m \colon 
\RR^2\times \{0\} \to \RR^2\times \{0\}$
(remember that 
$F_m$ 
denotes the first return map
of diffeomorphism cocycle  $\cF_m$ ).
 Then $\varphi$ projects to a diffeomorphism 
 of $T^{\infty}_{\cF_m}$.  Let us denote this projection by  $\tilde\varphi$. 

In some special case, we can also define 
the lift of the diffeomoprhisms on $\RR^2\times\{0\}$
to $T^{\infty}_{\cF_m}$.
To describe it, we prepare one notion.
The round circles centered at the 
origin in $\mathbb{R}^2 \times \{ 0 \}$ contained in the 
homothetic region of $\cF_m$ induce a foliation 
by parallels on $T^{\infty}_{\cF_m}$.   
We call each leaf of this foliation a \emph{round  parallel}. 
Then, we have the following
(remember that $\Gamma_{s, t}$ in the claim
denotes the annuls bounded by two circles centered 
at the origin with radii $0 <s <r$).

\begin{clm} \label{c.lifting}
Given any $\eta>0$ there is $\mu>0$ satisfying the following:
Let $\tilde{\varphi}$ be a diffeomorphism 
of $T_{\mathcal{F}_m}$ 
which satisfies
\begin{itemize}
\item The $C^1$-distance between $\tilde{\varphi}$ and
the identity map is less than $\mu$.
\item There is a round parallel 
disjoint from the support of $\tilde \varphi$.
\end{itemize} 
Then for any $m>0$ and any $r$ 
so that $\Ga_{\lambda^2r,r}$ 
is contained in the homothetic region of $F_m$, 
there exists a diffeomorphism $\varphi$, 
supported in a round fundamental domain 
contained in $\Ga_{\lambda^2r,r}$, 
whose projection on $T^{\infty}_{\cF_m}$ is $\tilde{\varphi}$ and 
is an $\eta$-perturbation of the identity map
(for the definition of $C^1$-distance on 
diffeomorphisms of $T^{\infty}_{\cF_m}$, 
see Remark \ref{r.disorb}).
\end{clm}\label{c.lift}
\begin{proof} The fact that the support of $\tilde \varphi$ is 
disjoint from one round parallel implies that
it admits a lift on some round fundamental 
domain 
$\Ga= \Ga_{\lambda r_0, r_0} \subset \mathbb{R}^2 
\times \{ 0 \}$ 
of $F_m$
in a homothetic region.  
Up to some homothetic conjugacy 
we can assume that $\Ga\subset \Ga_{\lambda^2r,r}$. 

Under this situation, one can easily see that 
there exists a (unique) lift $\tilde \varphi$ 
to $\RR^2\times \{ 0 \}$ supported in $\Ga$. 
Let us consider the $C^1$-distance between 
the identity map and $\varphi$.
Since the $C^1$-distance of $\varphi$ to the identity 
does not depend on the choice of the lift in the 
homothetic region 
we only need to consider a specific 
lift in a $\Ga_{\lambda^2r,r}$.
Since this correspondence $\tilde{\varphi} \mapsto \varphi$
is continuous and 
it sends the identity map on $T_F$ to the identity
map on $\mathbb{R}^2\times \{ 0 \}$, 
the choice of small $\mu$
endorses the closeness of the 
lifted diffeomorphism to the identity map.
\end{proof}

We perform a perturbation 
by composing such lifted maps to $\mathcal{F}_m$.
Let us see the effect of such perturbation.
First, for the $C^1$-distance, we have the following: 
given $\varphi$ supported in 
a round fundamental domain contained in the 
homothetic region of $F_m$, we denote 
by 
$\mathcal{F}_{m,\varphi} := \{f_{i, m , \varphi}\}$ 
the perturbation of 
the cocycle $F_m$ defined by
$f_{i,m,\varphi} :=f_{i,m}$ if $i\neq n-1$ 
and  $f_{n-1,m,\varphi} :=\varphi \circ f_{n-1,m}$. 
Then there is $C$ (depends only on $f_{n-1, m}$)
so that for every $m$, 
every $\eta>0$  and every $\varphi$ 
which is $\eta$-perturbation of the identity map, 
then $\mathcal{F}_{m,\varphi}$ is a $C\eta$-perturbation 
of the cocycle $\mathcal{F}_m$. 

For the behavior of the strong stable manifold, 
we have the following 
(Lemma \ref{l.fragmentation} below follows immediately from 
the definition):
\begin{lemm}\label{l.fragmentation} 
Let
$0<r_1<\lambda r_2<r_2<\lambda r_3\dots < r_k$ and 
$m >0$ given so that the round annulus 
$\Ga_{\lambda r_1,r_k}$ is contained in the 
homothetic region of $\mathcal{F}_m$.
Let $ \{\varphi_i\}$ $(i=1\dots,k)$ be diffeomorphisms 
on $\mathbb{R}^2 \times \{ 0 \}$ such that $\varphi_i$ is
supported in $\Ga_i= \Ga_{\lambda r_i,r_i}$,
and let $\Phi$ be the diffeomorphism which 
coincides with $ \varphi_i$ on $\Ga_i$ 
and equal to the identity outside $\Ga_{\lambda r_1,r_k}$. 
Then we have the following:
\begin{itemize}
 \item  $\mathcal{F}_{m,\Phi}$ is a contracting cocycle which coincides with $\mathcal{F}_m$ out of the homothetic region. 
 \item The meridians of $\mathcal{F}_{m,\Phi}$ are 
 $(\tilde \varphi_k)^{-1}\circ\dots\circ 
(\tilde\varphi_1)^{-1}(\gamma_i)$
\end{itemize}
\end{lemm}

Now let us perform the perturbation. 
Consider $\eta<\varepsilon_1/C$ and 
$\mu$ associated to $\eta$ by Claim~\ref{c.lifting}. 
The fragmentation lemma ensures that 
the diffeomorphism $\psi$ (for which $\psi(\gamma_i)=\sigma_i$) 
can be written as
$$\psi=(\tilde \varphi_k)^{-1}\circ\dots\circ (\tilde\varphi_1)^{-1}$$
where $k >0$ and $\tilde \varphi_i$ are 
diffeomorphisms of $T_F$ supported in small discs 
(so that each $\varphi_i$ has at least one round parallel 
disjoint from its support) and so that the $C^1$-distance 
from the identity is less that $\mu$. 

Then we fix $m>3(k+1)$ so that there is 
a round annulus $\Ga_{\lambda^m r,r}$ contained 
in the homothetic regions of $\mathcal{F}_m$. 
Then, each $\tilde\varphi_i$ admits a lift
$\varphi_i$ supported  in an annulus  
$\Ga_i = \Ga_{\lambda r_i,r_i}$ for 
some $\lambda^{m-3i+1} r\leq  r_i\leq \lambda^{m-3i} r$
such that the sequence $\{r_i\}$ satisfies the hypotheses 
of Lemma~\ref{l.fragmentation}. 
Therefore, $\cF_{m,\Phi}$
is the announced $\varepsilon_1$-perturbation 
where $\Phi$ is the diffeomorphism whch 
coincides with $ \varphi_i$ on $\Ga_i$ 
and equal to the identity out of the $\Ga_i$. 
\end{proof}

\begin{rema}\label{r.disorb}
In Claim \ref{c.lifting}, we did not specify the 
definition of $C^1$-distance put on the space of 
$C^1$-diffeomorphisms on $T_{\mathcal{F}}$.
In fact, as was elucidated in the proof, such a choice 
is not important for Claim \ref{c.lifting} and the whole proof. 

\end{rema}


\section{Construction of retardable cocycles}\label{s.realiza}

The aim of this section is the proof of Proposition~\ref{p.flexible-retardable}. 
Let  $\cA=\{A_i\}$ be an $\varepsilon$-flexible contracting 
linear cocycle, $i\in\ZZ/n\ZZ$. This gives us, 
by definition, a path $\{A_{i,t}\}$ of contracting linear cocycles.  
We will use this path for building a retardable contracting 
cocycle isotopic to $\cA$ with several other properties.
Our main tool is the Proposition \ref{p.realization} below,
which realizes paths of contracting linear cocycles as 
diffeomorphisms contracting cocycles. 
Note that Proposition \ref{p.realization} is independent of 
the notion of flexibility. 

We prepare one notation.
Let 
\[
\cC_n :=\mathrm{GL}(2,\RR)^n=\{\cA=\{A_i\} \mid 
A_i\in \mathrm{GL}(2,\RR), i\in\ZZ/n\ZZ\}
\] 
be the space of linear cocycles of period  $n$.
Remember that a cocycle is called contracting if 
the total space $\mathbb{R}^2 \times \mathbb{Z} / n \mathbb{Z}$
is contained in the basin of the orbit of the origin.
We denote by $\cC_{n, \mathrm{con}}\subset \cC_n$ 
the (open) subset of contracting cocycles. 

\begin{prop}\label{p.realization}
Let $\cO\subset \cC_{n, \mathrm{con}}$ 
be a relatively compact open subset. 
Then for any $\varepsilon_1>0$ there is 
$\delta>0$ with the following property:
Consider any $C^1$-path $\cA_t\colon [0,+\infty[\to \cC_n$,  $t\mapsto \{A_{i,t} \mid i\in\ZZ/n\ZZ\}$ which is constant near $t =0$. 
Assume that we have 
$$\left\|\frac{\partial A_{i,t}}{\partial t}\right\| \leq \frac{\delta}{t}.$$

Then the cocycle of maps
$\mathcal{F}=\{f_i \mid  \RR^2\times\{i\} \to\RR^2\times\{i\}, i   \in\ZZ/n\ZZ\}$ 
defined as 
$$f_i (x) := A_{i,\|x\|}(x),$$
satisfies the following:
\begin{itemize}
 \item $\mathcal{F}$ is a contracting diffeomorphisms contracting cocycle;
 \item At each point $(x,i)$ one has 
 $$\left\|Df_i(x)-A_{i,\|x\|}\right\|<\varepsilon_1.$$
\end{itemize}
\end{prop}

Let us first show how to
derive  Proposition~\ref{p.flexible-retardable}
from Proposition~\ref{p.realization}.

\subsection{Proof of  Proposition~\ref{p.flexible-retardable}} 

Let $\cA=\{A_i\}$ be a $\varepsilon$-flexible cocycle and 
$\cA_t=\{A_{i,t}, t\in[-1,1]\}$ be the path of 
cocycles in the definition of the $\varepsilon$-flexibility. 
Proposition~\ref{p.flexible-retardable} claims the 
existence of a contracting diffeomorphisms 
cocycle $\mathcal{F}$ coinciding with 
$\cA$ out of the unit ball, with the homothety on 
the $n$ first iterates of a round fundamental domain, and 
with $A_1$ in a neighborhood of the orbit of the origin. 
Furthermore, $\mathcal{F}$ needs to be isotopic to $\cA$ through 
contracting cocycles 
coinciding with $\cA$ out of the unit ball and having 
two different real positive eigenvalues. 

Recall that $\cA_t$  is a contracting cocycle for $t\neq 1$,
but $\cA_1$ is not a contracting eigenvalue. 
In order to deal with contracting property 
we will show the proposition but replacing $\cA_1$ 
by $\cA_{1-\eta}$ for some 
very small $\eta$: the corresponding diffeomorphisms cocycle will coincide with $\cA_{1-\eta}$ in a  
neighborhood of the periodic orbit and an extra small perturbation will change $\cA_{1-\eta}$ in $\cA_1$. 

The path $\cA_t, t\in[-1,1-\eta]$ is a compact segment 
in the open set of contracting linear cocycle. 
Therefore we can approximate it by a smooth path with the same properties: Thus we assume that 
$t\mapsto \cA_t$ is smooth. 
Recall that $\cA_{-1}$ is a homothety 
of ratio $\lambda<1$ in the period, 
$\cA_0=\cA$, and $\cA_t$ has two different 
real eigenvalues for $t \neq -1$. 
First, we reparametrize $\cA_t$ in order to apply Proposition~\ref{p.realization}. 

\begin{lemm}\label{l.realization}
Let $a(t) :  [0, 1] \to V$ be a smooth path in an 
Euclidean space $V$. 
Then, for every $\delta > 0$ there exists a smooth
non-decreasing function 
$\theta : [0, +\infty ) \to [0, 1]$ satisfying the following:
\begin{itemize}
\item $\theta (t) \equiv 0$ near $t=0$. 
\item $\theta (t) \equiv 1$ for  $t>1$. 
\item For every $t \in (0, +\infty )$, we have the following inequality:
\[
\left\| \frac{d(a \circ \theta)}{dt}(t) \right\| < \frac{ \delta}{t}.
\]
\end{itemize}
\end{lemm}

\begin{proof}
Just note that the length of the path $a(t)$ is finite 
while the integral $\int_0^1\delta t^{-1} dt$ is infinite for $\delta >0$.
\end{proof}

Let us start the proof of Proposition \ref{p.flexible-retardable}.

\begin{proof}[Proof of Proposition \ref{p.flexible-retardable}]

Let $\cA=\{A_i\}$ be an $\varepsilon$-flexible cocycle.
First, we fix $\varepsilon_1$ sufficiently small
so that $\varepsilon_1 + \mathrm{dist}(\cA_t) < \varepsilon$ holds.
We also fix small $\eta > 0$.  Precise choice of $\eta$ is fixed at the 
end of this proof.
By applying Lemma \ref{l.realization},
we reparametrize the path 
$\cA_t, t\in [-1,1-\eta]$ by a function $\theta\colon[0,+\infty[\to[-1,1-\eta]$ so that we have the following:
\begin{itemize}
 \item  For $t >0$ we have the following inequality:
$$\left\|\frac{\partial A_{i,\theta(t)}}{\partial t}\right\| \leq \frac{\delta}{t},$$
 \item $\theta(t)=0$ for $t\geq 1$,
 \item $\theta(t)=1-\eta$ near $t =0$,
 \item There are $0 < t_3< t_2 < t_1<t_0 <1$ such that:
\begin{itemize}
\item We have $\theta(t)=-1$ for $t\in[t_3,t_0]$.
 \item For $\{ t_i\}$ we have the following inequalities: 
$$ t_2<\lambda t_1, \quad t_3< K^{-n} t_2<t_1 K^n < t_0$$
where $n$ is the period of cocycle, 
$K := \max \{ \|A_{i,t}^{\pm 1}\|  \} \geq 1$ 
and $\lambda$ is the rate of the 
contraction of the homothety $A_{-1}$.
\end{itemize}
\end{itemize}

Then the announced cocycle $\mathcal{F} = \{ f_i\}$ is defined by 
$f_i(x) := A_{i, \theta(\|x\|)}(x)$. 
Indeed, the Proposition \ref{p.realization}, together with the 
first condition on $\theta$
implies the contraction property of $\mathcal{F}$.
The last condition ensures that this cocycle (in the period) is 
an homothety of ratio $\lambda$ on at 
least one fundamental domain, implying the retardable property.
More precisely, by choosing $R_1 = t_3$, 
$R_2 =t_1$  and $R_3 = t_0$, we can check the 
retardable property. 
Note that by the choice of $\varepsilon_1$,
we can deduce that the diffeomorphism cocycle $\mathcal{F}$
itself is an $\varepsilon$-perturbation of $\{A_i\}$. Note that
Lemma \ref{l.retper} implies that its retarded cocycles are 
also $\varepsilon$-perturbation.

Furthermore, the cocycle $\cF$ is isotopic 
to $\cA$ through contracting cocycles 
which coincide with $\cA$ 
out of the unit ball, and whose periodic orbit has 
two distinct real positive eigenvalues: 
for that it is enough to change $\theta$ 
by $\theta_s(t)=s\theta(t)$ for $s\in[0,1]$ (note that 
for every $s\in[0,1]$ we can apply Proposition \ref{p.realization}). 

To finish the proof, it remains to perform an extra perturbation 
in a very small neighborhood of the periodic point for 
turning the weakest eigenvalue of 
$\cA_{1-\eta}$ to be equal to $1$ preserving 
the contracting property of the cocycle. 
We can see that if $\eta$ is sufficiently small, then such a
perturbation can be attained in the form of isotopy. 
More precisely, first we perform an perturbation so that the local 
dynamics along the periodic orbit exhibits an eigenvalue-one direction. 
Then we add another perturbation so that the central direction 
turns to be topologically attracting, keeping the eigenvalue.

Thus the proof is completed.
\end{proof}

\subsection{Realizing a path of linear cocycles 
as a diffeomorphisms cocycle: 
proof of Proposition~\ref{p.realization}}

Let us start the proof of Proposition~\ref{p.realization}.
We consider a relatively compact open 
subset $\mathcal{O} \subset \cO_{n, \mathrm{con}}$, 
(remember 
that $\cO_{n, \mathrm{con}}$
is the space of contracting linear cocycles of period $n$). 
Before starting the proof, let us have some auxiliary observations.

We put 
$$K_\cO :=\max\{\|A_i^{\pm 1}\| \mid \cA=\{A_i\}_{i\in\ZZ/n\ZZ}\in\cO\} \geq 1,$$
that is, the bound of matrices and 
their inverse  for  the cocycles in $\cO$.

The relative compactness of $\cO$
(compactness of the closure) implies that 
this bound is finite.
The relative compactness of $\mathcal{O}$, together with 
the fact that each cocycle in $\cO$ is 
contracting, implies that their are uniformly contracting
in the following sense:
\begin{lemm} Let $\cO$ be a relatively compact set 
of contracting linear cocycles.  
Then there is $k_\cO>0$ such that for every $\cA=\{A_i\}\in \cO$ 
and for every $i\in\ZZ/n\ZZ$ we have 
 $$\|A_{i+k_{\mathcal{O}}-1}\circ\cdots\circ A_i\|<\frac 12.$$
\end{lemm}

\begin{rema}In fact, we will prove that  
the number $\delta$ announced by 
Proposition~\ref{p.realization} only depends on 
$\varepsilon$,  $K_\cO$ and $k_\cO$ 
and is independent of the period $n$ and 
of the relatively compact set $\cO$. 
\end{rema}

Note that the relative compactness of $\mathcal{O}$
also implies that $\overline{\mathcal{O}}$ does not contain 
any singular matrices. This fact, combined with compactness 
argument yields the following:
\begin{lemm}\label{l.inverti}
Given a relatively compact set 
$\mathcal{O} \subset \mathcal{C}_{n,\mathrm{con}}$,
there exists $\mu_{O} > 0$  such that for every 
$\{ A_{i} \} \in \mathcal{O}$, if $\{ B_i \} \in M(2,\RR)^n$ 
(where $M(2,\RR)$ is a set of square matrices of size $2$)
satisfies $\| A_i -B_i \| < \mu_{\mathcal{O}}$ then we have 
$\{ B_i \} \in \mathrm{GL}(2,\RR)^n$.
\end{lemm}

\begin{rema}\label{r.inverti}
Let $K >1$. Then, the following set
$$
\mathcal{B}_K := \{\{A_i\} \in  \mathrm{GL}(2,\RR)^n  \mid \max\{\|A_i^{\pm 1}\| \} \leq K\}
$$
is a compact set. Thus, we can apply Lemma \ref{l.inverti} to 
$\mathcal{B}_K $. We denote corresponding $\mu$ by $\mu_K$.
\end{rema}

Then, we prove the following.
\begin{prop}\label{p.realization1} Given $K\geq 1$, 
$ \varepsilon_1 >0$ and 
an integer $k>0$ there exists $\delta>0$ 
such that the following holds:
Given any $n>0$ and any path $\cA_t=\{A_{i,t}\}_{i\in\ZZ/n\ZZ}$, $(t\in[0,1])$
satisfying:
\begin{itemize}
 \item $\|A_{i,t}^{\pm 1}\|<K$ for every $i$ and $t$, 
 \item $\|A_{i+k-1,t}\circ\cdots\circ A_{i,t}\|<1/2$ for every $i$ and $t$,
 \item $\displaystyle \left\|\frac{\partial A_{i,t}}{\partial t}\right\| \leq \frac{\delta}{t}$.
\end{itemize}

Then the diffeomoprhims cocycle 
$\mathcal{F} =\{F_i\}_{i\in\ZZ/n\ZZ}$, 
where $f_i\colon \RR^2\times\{i\} \to\RR^2\times\{i +1\}$ are define as 
$$f_i(x) := A_{i,\|x\|}(x)$$
satisfies:
\begin{itemize}
 \item $\cF$ is a contracting diffeomorphisms contracting cocycle,
 \item For each point $(x,i) \in \RR^2 \times \ZZ / n\ZZ$ and 
for every $0\leq j\leq k-1$, we have 
 $$\left\|D\cF^j(x,i)-A_{i+j-1,\|x\|}\circ\cdots\circ A_{i,\|x\|}\right\|<\varepsilon_1.$$
(Remember that $\mathcal{F}$ is the diffeomorphism of total space
$\RR^2 \times \ZZ / n\ZZ$.)
\end{itemize}
\end{prop}
\begin{proof}
First, note that the inequality in the
second item implies the contracting property: 
the second item implies that 
$\|D\mathcal{F}^k\|<1/2 +\varepsilon_1$, thus by exchanging $\varepsilon_1$
with a number smaller than $1/2$ if necessary, we obtain 
$\|DF^k\|<1$. 

Let us start proving this inequality
for $j=1$. By direct calculation, for every $i$, we have
$$Df_i(x)= D(A_{i,\|x\|})(x,i)= A_{i,\|x\|}+ 
 \left( \frac{d A_{i,t}}{dt}\Big|_{t=\|x\|} \right) \otimes
\left( \frac{ \partial (\|x\|)}{\partial x} \right) (x),$$ 
Therefore, we have
$$\|Df_i(x)-A_{i,\|x\|}\|\leq C \|x\| \cdot \left\| \frac{d A_{i,\| x\|}}{dt} \right\|,$$ 
where $C$ is a constant which does not depend on particular choice 
of other constants.

Then, we fix $\delta$ so that 
$\delta<\min \{\varepsilon_1, \mu_K \}/C$ holds (for the definition 
of $\mu_K$ see Lemma \ref{l.inverti} and Remark \ref{r.inverti}).
This guarantees that  
$\left\|\frac{\partial A_{i,t}}{\partial t}\right\| \leq\frac{\delta}{t}< \min \{\varepsilon_1, \mu_K \}/(Ct)$.
So we have: 
$$\|Df_i(x)-A_{i,\|x\|}\|\leq \|x\|. \frac{C\delta}{\|x\|}=C\delta<\min \{\varepsilon_1, \mu_K \},$$ 
A priori, each
which implies the desired inequality for $j=1$.
Furthermore, this inequality, together with Lemma \ref{l.inverti},
shows that $Df_i : \mathbb{R}^2 \times \{ i \} \to 
\mathbb{R}^2 \times \{ i+1 \}$ is a local diffeomorphism.
This fact, combined with some topological observation, concludes 
that $f_i$ is a diffeomorphism for every $i$.

Now we start the proof of inequality above for general $j < k$.
The difficulty is the following:
The differential $D\mathcal{F}^j(x,i)$ is the product 
$$D\mathcal{F}^j(x,i)= D\cF(\cF^{j-1}(x,i) )\circ\cdots\circ D\cF(x,i).$$
A priori, the distance between $\cF^{j-1}(x,i)$ and $(x, i+j)$ can be 
very big.  Thus the corresponding differentiations can be very different.
Our strategy is to choose sufficiently small 
$\delta$ so that the matrix 
$A_{i+\ell,\|F^\ell(x,i)\|}$ remains almost 
equal to $A_{i+\ell,\|x\|}$, for $0\leq \ell\leq k-1$.

We start from bounding $\|\mathcal{F}^\ell(x,i)\|$ for $0\leq\ell<k-1$ as follows: 
$$ \frac{\|x\|}{K^k}  \leq \|F^\ell(x,i)\|\leq K^k \|x\|,$$
where $K$ is the uniform bound of the norms $\|A^{\pm1}_{i,t}\|$.
Then, a simple argument involving the mean value theorem implies: 
\begin{clm}\label{cl.meanvalue} For any $\nu>0$ there is $\delta>0$ such that if
$\displaystyle \left\|\frac{\partial A_{i,t}}{\partial t}\right\| \leq \frac{\delta}{t}$, 
for every $t >0$, 
then for every $t >0$, $i \in \ZZ /n\ZZ$, $\ell$ satisfying $0\leq \ell<k$ and
$s \in [ K^{-k}t,K^k t] $ we have
$$\|A_{i+\ell, t}-A_{i+\ell, s}\|<\nu. $$ 
\end{clm}
Also a simple compactness argument together with the continuity of product of matrices shows the following.
\begin{clm}\label{cl.continuity}
There is  $\nu>0$ so that for any matrices 
$\{B_{i,t}\}_{i\in\ZZ/n\ZZ,t\in[0,1]}$
with $\|B_{i,t}-A_{i,t}\|<2\nu$, any 
$0\leq j <k$ and for any $i, t, \varepsilon_1 >0$ we have:

$$\left\|B_{i+j-1,t}\circ\cdots\circ B_{i,t}-  
A_{i+j-1,t}\circ\cdots\circ A_{i,t}\right\|<\varepsilon_1.$$
\end{clm}
Now we are ready for proving the inequality.
We fix $\nu>0$ given by Claim~\ref{cl.continuity} and 
$\delta$ given by Claim~\ref{cl.meanvalue}.
Then for every $0\leq \ell<k$ we have 
\begin{align*}
 \left\|D\cF(\cF^\ell(x,i))-A_{i+\ell,\|x\|}\right\|&
 \leq \left\|D\cF(\cF^\ell(x,i))-A_{i+\ell,\|\cF^\ell(x,i)\|}\right\|+ 
 \left\|A_{i+\ell,\|\cF^\ell(x,i)\|}-A_{i+\ell,\|x\|}\right\|\\
 &\leq \nu+\nu=2\nu.
\end{align*}
The choice of $\nu$ implies the announced inequality
$$\left\|D\cF^j(x,i)-A_{i+j-1,\|x\|}\circ\cdots\circ A_{i,\|x\|}\right\|<\varepsilon_1.$$
Thus the proof is completed.
\end{proof}


\section{Abundance of flexible periodic points}\label{s.abunda}

The aim of this section is the proof of Theorem~\ref{t.aboundance}  
(and therefore of Theorem~\ref{t.generic}). 

The proof contains two steps. 
The first step is showing that the hypotheses of 
Theorem~\ref{t.aboundance} lead 
(up to arbitrarily small perturbations) to the coexistence 
of points with complex stable eigenvalues and points with a stable 
eigenvalue arbitrarily close to $1$ in the same basic set.  
The second one is showing that such basic sets 
contain flexible points. 

In this section,  $M$ denotes a smooth 
compact manifold endowed with a Riemannian metric.

\subsection{Periodic points with arbitrarily weak stable eigenvalue}

Let us start the proof of first step. 
We show the following.
\begin{prop}\label{p.weak} 
Let $f \in \mathrm{Diff}^1(M)$
with a hyperbolic periodic point $p$
and an open neighborhood $U$ of the 
orbit  $\cO(p)$  with the following property 
\begin{itemize}
\item $p$ has stable index equal to $2$. 
 \item there is a stable index $1$
 periodic point $q$ with $\cO(q)\subset U$.
 \item there are  hyperbolic transitive  basic sets $K\subset H(p,U)$ and $L\subset H(q,U)$ 
 containing $p$ and $q$, respectively, so that $K,L$ form a $C^1$-robust heterodimensioanl cycle in $U$
(that is, the orbits of $K$ and $L$ are contained in $U$ 
and there exists heteroclinic points between 
$W^u(K)$ and $W^s(L)$, $W^u(L)$ and $W^s(K)$ whose 
orbit is contained in $U$).
\end{itemize}

then, for every $\nu>0$  there are $g$ arbitrarily $C^1$-close to $f$ having a periodic point $p_2$ with the following properties
\begin{itemize}
 \item $p_2$ has stable index $2$ and  is homoclinically related with $p$ in $U$;
 \item  $p_2$ has a real stable eigenvalue $\lambda^{cs}(p_2)$ with $|\lambda^{cs}|\in [1-\nu,1)$;
 \item $p_2$ has the smallest  Lyapunov exponent so that 
 $$\chi^{ss}(p_2)\in [ \inf\{\chi^{ss}(p),\chi^{ss}(q)\}-\nu,
 \sup\{\chi^{ss}(p),\chi^{ss}(q)\}+\nu] ,$$
 \item the orbit of $p_2$ is $\nu$-dense in the relative homoclinic class $H(p,U,f)$ for $f$.  
\end{itemize}
\end{prop}

\begin{proof}[Sketch of proof of Proposition~\ref{p.weak}]
The creation of periodic orbits with eigenvalues arbitrarily close to $1$ 
inside a homoclinic class containing a robust heterodimensional 
cycle has been already done in 
\cite{ABCDW}. The proof of Proposition~\ref{p.weak} can be 
carried out in the similar fashion. So we only show the sketch of the 
proof of it.

We fix $\nu > 0$. First,
note that $K$ or $L$ is non trivial since otherwise they 
cannot have robust heterodimensional cycle.
Thus by performing perturbation by 
Hayashi's connecting lemma if necessary, 
we can assume that both of them are not trivial.

Recall that the homoclinic class of $p$ is the 
closure of transverse homoclinic intersection,
hence is the Hausdorff limit of an 
increasing sequence of hyperbolic basic sets. 
The corresponding fact is also true for relative homoclinic classes. 
Therefore one can choose a hyperbolic basic 
set $\tilde K\subset U$ whose Hausdorff distance 
with $H(p,U,f)$ is less that $\nu/10$.
We can also choose an arbitrarily small 
perturbation $f_0$ of $f$ and a periodic point 
$\tilde p\in \tilde K(f_0)$ where $\tilde K(f_0)$ 
is the hyperbolic continuation of $\tilde K$
so that
\begin{itemize}
 \item 
the Hausdorff distance between the orbit $\cO(\tilde p)$ 
and $H(p, U, f_0)$ is less that  $\nu/5$, 
\item $|\chi^{ss}(\tilde p, f_0)-\chi^{ss}(p,f)|<\nu/10$.
\item the Lyapunov exponent $\chi^{ss}(\tilde{p}, f_0)$ has multiplicity $1$, that is, the restriction of  the derivative to the stable 
plane  has two distinct real eigenvalues 
(see section 2 of \cite{ABCDW} or section 4 of \cite{BCDG}).
\end{itemize}

If the perturbation $f_0$ is sufficiently 
close to $f$, then one still has a $C^1$-robust 
heterodimentional cycle associated with $K(f_0)$ and $L(f_0)$ 
(and therefore $\tilde K(f_0)$ with $L(g_0)$).
In other words, by replacing $f$ with $f_0$ and 
$\nu$ with $\nu/2$ , one may assume that
\begin{itemize}
 \item  the orbit of $p$ is $\nu/2$ dense in $H(p,U,f)$
 \item $p$ has two real distinct eigenvalues
\end{itemize}

In the same way, by changing $q$ 
with another periodic point in $L$ and 
performing an arbitrarily 
small perturbation of $f$,  
one may assume that $q$ has 
a real weakest unstable eigenvalue.   

Then, we perform second perturbation
to construct a heterodimensional cycle 
between $p$ to $q$ in $U$ as follows 
(see section 2 of  \cite{ABCDW}):
\begin{itemize}
 \item as $p$ and $q$ belong to the same chain recurrence class, 
an arbitrarily small perturbation 
 (using for instance  the connecting lemma in \cite{BC}) allows 
to create a transverse intersection between 
 $W^u(q)$ and $W^s(p)$.
 \item  as the $C^1$-robust cycle persits under the first
 perturbation, $p$ and $q$ still belong to the same class
 so that an arbitrarily small perturbation, preserving the first
 intersection, allows to create a transverse intersection between 
 $W^s(q)$ and $W^u(p)$. 
\end{itemize}

Now \cite{ABCDW} (see section 3 of \cite{ABCDW},
see also \cite{BD2}, \cite{BDK}) 
tells us that 
by performing an arbitrarily small perturbation 
to the heterodimensional cycle as above, 
we can create a periodic point $p_2$ with the following property
(we denote the perturbed diffeomorphism by $g$): 
\begin{itemize}
 \item $p_2$ has a stable index $2$
 \item $p_2$ has a weakest stable eigenvalue $\lambda^{cs}$ 
with absolute value $|\lambda^{cs}| = 1-\nu <1$
 \item the orbit of $p_2$ passes arbitrarily close to $p_g$ 
 (as a consequence $\cO(p_2)$ will be $\nu/2$ dense in $H(p,U,f)$
 \item $\chi^{ss}(p_2)$ is arbitrarily close to a convex sum of $\chi^{ss}(p, f)$ and $\chi^{ss}(q, f)$
 \item the unstable manifold of $p_2$ cuts transversely 
the stable manifold of $p$ and the stable manifold of 
 $p_2$  cuts transversely the unstable of $q$. 
\end{itemize}

The last item implies that $p_2$ and  $p$  are  robustly 
in the same chain recurrence class (since we can always 
find a pseudo-orbit from $q$ to $p$ following the robust 
heterodimensional cycle between $K$ and $L$): 
a new arbitrarily small perturbation by connecting lemma 
in \cite{BC} creates 
a transverse intersection between the stable manifold of $p$
with the unstable of $q$, ending the proof.

\end{proof}


\subsection{Weak eigenvalues, complex eigenvalues, and flexible points}

The aim of the remaining part of the section is the proof of the next proposition. 

\begin{prop}\label{p.weak-flexible}
Given $C>1$, $\chi<0$ and $\varepsilon>0$, 
there exists $\nu \in (0, 1)$ with the following property:
Let $f \in \mathrm{Diff}^1(M)$ be a diffeomorphism and 
$\Lambda$ be an compact invariant hyperbolic basic set 
of $f$ with stable index two
such that $\|Df\|$ and $\|Df^{-1}\|$ are bounded by $C$ over $\Lambda$. 

Suppose that $\Lambda$ contains a periodic point $q$
(of stable index two)
 with complex (non-real) stable eigenvalue and
a point $p$ having two distinct real stable eigevalues so that
\begin{itemize}
 \item  the smallest Lyapunov exponent of $p$ is less than $\chi<0$
 \item the stable eigenvalue with larger absolute value
 $\lambda^{cs}$ satisfies $|\lambda^{cs}|\in (1-\nu,1)$.
\end{itemize}

Then $f$ admits arbitrarily small perturbations with 
$\varepsilon$-flexible points contained the continuation of $\Lambda$
and the $\varepsilon$-neighborhood of $\mathcal{O}(x_L)$
contains $\mathcal{O}(p)$ in the Hausdorff topology.
\end{prop}

This proposition, together with the previous Proposition
\ref{p.weak} with standard genericity argument
(involving the generic continuity of the homoclinic classes 
with respect to the Hausdorff distance)
implies Theorem~\ref{t.aboundance}.

The main ingredient of the proof is the following fact:
in a basic sets, given a set of finite number of periodic points,
one may choose a periodic point which travels around these 
periodic points with the predetermined itinerary. 
By choosing a convenient itinerary, we can find a periodic 
point whose differential behaves in the way which is
very close to what we want. Thus,  by adding some perturbation, 
we can obtain the desired orbit. 
This technique has been formalized in \cite{BDP} by 
the notion of \emph{transition}.

\subsection{Transitions on the periodic points}\label{sb_transitions}

In this subsection we will extract 
some consequence from \cite{BDP} 
which will be useful for us. 
For the proof of Lemma \ref{l.transition1}, 
see Lemma 1.9 of  \cite{BDP} (indeed, 
Lemma \ref{l.transition1} is just the special 
case of Lemma 1.9 of  \cite{BDP}).

\begin{lemm}\label{l.transition1}
Suppose $f \in \mathrm{Diff}^1(M)$ have  hyperbolic basic set $\La$
with stable index $2$. Let $T\Lambda = E^s \oplus E^u$ 
be the hyperbolic splitting (thus $\dim E^s =2$). 
We fix coordinate on $\La|_{E^s}$ so that
and take the matrix representation of $df$. 
Let  $x_1,x_2\in \La$ be two hyperbolic periodic saddle points 
of period  $\pi_i$ $(i = 1, 2)$, respectively. 

Then, given $\varepsilon_2 >0$, there exists two finite sequence of 
matrices in $\mathrm{GL}(2, \mathbb{R})$
$(T_i^j)$ ($j=0,\ldots, j_i-1$)  ($i = 1, 2$) (where $j$ denotes the suffix, not the 
power of the matrix), with the following property: 

For any  $L =(l_1,l_2,l_3,l_4)\in \mathbb{N}^4$
satisfying $(l_1, l_2) \neq (l_3, l_4)$, 
there exists a periodic point $x_L \in \Lambda$ such that the following:
\begin{itemize}
\item The period of $x_L$ is $(l_1 +l_3)\pi_1 + (l_2 +l_4) \pi_2 + 2(j_1 +j_2)$.
\item If $k=K$ with $0 \leq K < \pi_1 l_1$ or 
$k=l_1 \pi_1 + l_2 \pi_2 + j_1 +j_2 + K$ with $0 \leq K < \pi_1 l_3$, then 
$Df|_{E^{s}}(f^k(x_L))$ is $\varepsilon_2$-close to 
$Df|_{E^{s}}(f^{k}(x_1))$.
\item If $k= l_1\pi_1+ j_1+ K$ with $0 \leq K < \pi_2 l_2$ or 
$k=(l_1+l_3) \pi_1 + l_2 \pi_2 + 2j_1  + K$
with $0 \leq K < \pi_2 l_4$,  then
$Df|_{E^{s}}$ is $\varepsilon_2$-close to 
$Df|_{E^s}(f^{K}(x_2))$.
\item If $k= l_1\pi_1 + K$ or
$k= (l_1+l_3)\pi_1 + l_2\pi_2 + K$ with 
$0 \leq K < j_1$,  then
$Df|_{E^{s}}(f^k(x_L))$ is $\varepsilon_2$-close to $T_1^{K}$.
\item If $k= l_1\pi_1 + l_2\pi_2 + j_1  + K$ or
$k= (l_1+l_3)\pi_1 + (l_2+l_4)\pi_2 + 2j_1 +j_2 + K$ 
with $0 \leq K < j_2$, then 
$Df|_{E^{s}}(f^k(x_L))$ is $\varepsilon_2$-close to $T_2^{K}$.
\end{itemize}
We put $T_i := \prod_{j=0}^{j_i-1} T_{i}^{j}$ and call them \emph{transition matrices}. 
\end{lemm}

 \begin{rema}\label{r.distance}
 In the above Lemma, by adjusting $L$ we can control the 
 position of the orbit of $x_L$. More precisely,
 if we take $l_1$ or $l_3$ (resp. $l_2$ or $l_4$) very large, 
 then $x_L$ passes arbitrarily close to $x_1$ (resp. $x_2$).
 See {\normalfont \cite{BDP}} for detail.
 \end{rema}

\subsection{Rudimentary results from linear algebra}\label{sb_linear}
We collect two results from linear algebra, which will be 
used in the proof of Proposition \ref{p.weak-flexible}.

First, we prove the following lemma.
\begin{lemm}\label{l.annihi}
Let $T \in \mathrm{GL}_{+}(2, \mathbb{R})$
and $Q$ be a contracting homothety (i.e., $Q = \lambda \mathrm{Id}$
where $\lambda$ satisfying $0 < \lambda < 1$).
Then, given $\varepsilon > 0$
exists $h>0$ such that the sequence of matrix $(J_i)$  (resp. $(L_i)$)
($i = 0, \ldots, h-1 $) such that 
\begin{itemize}
\item each $J_i$ (resp. $L_i$) is $\varepsilon$-close to $Q$.
\item The product $T (\prod_{i=0}^{h-1} J_i)$ 
(resp. $(\prod_{i=0}^{h-1} L_i) T$ ) is a contracting homothety.
\end{itemize}
\end{lemm}

\begin{proof}
We only give the proof of the existence of $(J_i)$. 
The proof of $(L_i)$ can be carried out similarly.

Let $T$, $Q$, and $\varepsilon$ be given. 
First, given $Q$, we fix $\delta >0$ such that the following 
holds: If $X \in \mathrm{GL}_{+}(2, \mathbb{R})$ is 
$\delta$-close to $\mathrm{Id}$, then $Q$ and $XQ$ are 
$\varepsilon$-close. We can fix such $\delta$ because of the 
continuity of the multiplication.

For every $T \in \mathrm{GL}_{+}(2, \mathbb{R})$, 
there is a continuous path $I(t)$ in $\mathrm{GL}_{+}(2, \mathbb{R})$
such that $I(0) = T$ and $I(1) = \mathrm{Id}$
(since $\mathrm{GL}_{+}(2, \mathbb{R})$ is path-connected).
Then, because of the compactness of the path, 
we can take a sufficiently large integer $m >0$ such that 
the following holds: for every $k$  $(0 \leq k < m-1$),
$I((k+1)/m)(I(k/m))^{-1}$ is $\delta$-close to the identity.
We put $h = m$ and $J_k = I((k+1)/m) \cdot (I(k/m))^{-1}Q$.
Since $Q$ is a homothethy, we have
\begin{eqnarray*}
T \prod_{k=0}^{h-1}(J_k) 
&=& T\prod_{k=0}^{h-1} \left( I((k+1)/h) \cdot (I(k/h))^{-1}Q \right)  \\
&=& TQ^h \prod_{k=0}^{h-1} I((k+1)/h\cdot(I(k/h))^{-1}  \\
&=& TQ^h \cdot I (1) \cdot (I (0))^{-1} = Q^h.   
\end{eqnarray*}
This completes the proof.
\end{proof}

Let us see the second Lemma.
By $R(\theta)$ we denote the rotation matrix of angle $\theta$,
more precisely, we put
\[
R(\theta) := 
\begin{pmatrix}
\cos \theta & - \sin \theta \\
\sin \theta              &  \cos \theta
\end{pmatrix}.
\]

We need to check the following:
\begin{lemm}\label{l.homoper}
Let $1> \lambda_1>\lambda_2 >0$. Then, for $0<t<1$, 
The matrix
\[
M(t) :=
R(-t)
\begin{pmatrix}
\lambda_1 & 0 \\
0              &  \lambda_2
\end{pmatrix}
R(t)
\begin{pmatrix}
\lambda_1 & 0 \\
0              &  \lambda_2
\end{pmatrix}
\]
satisfies the following properties:
\begin{itemize}
\item For $0 \leq t < \pi /2$, $M(t)$ has two distinct 
positive contracting eigenvalues.
\item For $t = \pi /2$, $M(t)$ is a homothetic contraction.
\end{itemize}
\end{lemm}
\begin{proof}
Both of items can be checked by calculating the characteristic 
polynomial of $M(t)$ directly. Indeed, it is given by 
$x^2 -\mathrm{tr}(M(t)) x + \det (M(t))$.
Note that $\det (M(t))$ is equal to $(\lambda_1\lambda_2)^2$ 
(independent of $t$). By a direct calculation, we can see 
that  $\mathrm{tr}(M(t))$ is equal to
\[
(\lambda_1^2 +\lambda_2^2)\cos^2 t + 2\lambda_1 \lambda_2\sin^2 t.
\]
One can check that  this value is monotone
decreasing on $[0, \pi /2]$.
Then, by a direct calculation we can check
the desired properties of the path above. 
\end{proof}

\subsection{Proof of Proposition~\ref{p.weak-flexible}}\label{sb_th3}

Let us start the proof of Proposition~\ref{p.weak-flexible}.
For that we use Lemma~\ref{l.homoper}, which explains 
the behavior of products of rotation matrices and diagonal 
matrices. To use it, we need to take convenient bases on 
tangent spaces of our basic set. First let us see 
how to fix such bases.

\subsubsection{Diagonalizing coordinates}
In  the assumption of Proposition~\ref{p.weak-flexible},
the basic set $\La$ contains a periodic point $p$ with 
distinct real eigenvalues: 
one eignevalue is very close to $\pm 1$ but the smallest 
Lyapunov exponent is bounded away from $0$.  
In particular, the restriction of the differential $Df^\pi(p)$,
where $\pi$ is the period of $p$,  
to the stable plane is diagonalisable. 
We want to take the pair of eigenvectors as 
the basis of stable tangent directions along the orbit of $p$, 
but such a coordinate change may be big (with respect to 
the metric induced from the original Riemannian metric)
and the seemingly small perturbations looked through 
the diagonalized coordinate can be very big.
Thus, let us consider the size of such coordinate change.

To clarify our argument, we generalize the situation as follows.
Let $A_i : V_i \to V_{i+1}$ $(i \in \mathbb{Z}/\pi \mathbb{Z})$ 
be a sequence of isomorphims between the sequence of 
two dimensional Euclidean vector spaces $V_i$.
Suppose the product $A_{\pi-1}  \cdots  A_0$ is diagonalizable.
Thus, on each $V_i$, there are images of two
eigenspaces 
$U_i$ and $W_i$ of dimension one.
We assume that $U_i$ is the strong contracting 
direction.

For each $V_i$,  we fix 
an orthonormal basis $\langle e_{1, i}, e_{2, i} \rangle$
so that $U_i \wedge W_i$ and $(e_{1, i})\wedge (e_{2, i})$
defines the same orientation. Then, 
let $B_i\in \mathrm{SL}(2,\RR)$ be the matrix 
representing the change of the basis  from the orthonormal 
one to the one having one unit vector on the most 
contracted eigendirection.
More precisely, $B_i:V_i \to V_i $ is the (unique) matrix in $\mathrm{SL}(2,\RR)$   
satisfying $B_i(e_{1, i}) \in U_i$ and $\| B_i(e_{1, i}) \| =1$
(we take the matrix representation regarding 
$\langle e_{1, i}, e_{2, i} \rangle$ to be the standard basis).
Up to multiplication by an orhtogonal matrix
(derives from the ambiguity of the choice of initial orthonormal 
basis), 
$B_i$ is well defined. In particular, $\|B_i\|$ is well defined. 

The value $\|B_i\|$ measures the angle between 
two eigendirections on $V_i$
in the sense that
it is a strictly decreasing function of the angle.
The norm $\|B_i\|$ is equal to $1$ when two eigendirections 
are orthogonal
and diverges to $+\infty$ as the angle tends to $0$
(this is easy to prove, so we omit it).



Now, we use the following Lemma from linear algebra:
\begin{lemm}\label{l.diag}
For every $C_1>1$ and $\chi <0$, 
there exists $\alpha>0$ 
such that the following holds: 
Let $A_i : V_i \to V_{i+1}$ $(i \in \mathbb{Z}/\pi \mathbb{Z})$ 
be a sequence of isomorphims between 
two dimensional Euclidean vector spaces
with norms $\|A_i^{\pm}\|<C_1$.
Suppose the product $ A_{\pi-1}\cdots A_0$ is diagonalizable and 
the Lyapunov exponents of it satisfy the following:
\begin{itemize}
\item The smaller one is less than $\chi$.
\item The larger one is greater than $\chi/2$.
\end{itemize}
then, there is $i\in\ZZ/\pi\ZZ$ on which the 
matrix of coordinate change $B_i$ (defined as above)
has a norm $\|B_i\|$ smaller than $\alpha$. 
\end{lemm}

We apply this Lemma \ref{l.diag} to our situation
letting $C_1 = C$, $\chi$ given in the hypotheses 
of Proposition \ref{p.weak-flexible},
$V_i = T_{f^{i}(p)}\Lambda|_{E^s}$ and $A_i = Df(f^i(x))|_{E^s}$. 
It implies that there is a constant $\alpha$  
depending only on $C_1, \chi$ 
which there is at least a point $f^i(p)$ of the orbit of $x$ 
where $\|B_i\|<\alpha$.  
Thus by replacing $p$ with $f^i(p)$ 
we can assume $\|B_0\|<\alpha$. 
In other words, up to a conjugacy by matrices 
in $\mathrm{SL}(2,\RR)$ bounded in norm by $\alpha$, 
one may assume that $Df^\pi(x)$ is diagonal. 
This bounded change of coordinates induces 
a bounded change of the notion of perturbations
(remember that, for matrices in $SL(2,\RR)$ the norm of the matrix 
and of its inverse are the same).  

Thus, up to multiplication by some constant, we can assume that
a $\delta$-perturbation with respect to this coordinate 
is also a $\delta$-perturbation to the orthonormal coordinate.
So we fix some coordinate near
$T_{f^{i}(p)}\Lambda|_{E^{s}}$ explained as above 
and continue the proof.

Let us prove the Lemma \ref{l.diag}. First, note that 
a simple compactness argument shows the following.
\begin{lemm}\label{l.distor}
Let $C_1> 1$. Then, for every $\kappa > 1$ there exists 
$\tau >0$ such that the following holds:
For every $A \in \mathrm{GL}(2, \mathbb{R})$ 
satisfying $\|A\|, \|A^{-1}\| < C_1$, if $u, v$ are unit vector 
such that the angle between them is less than $\tau$, 
then $1/\kappa<\|Au\|/\|Av\| <\kappa$.
\end{lemm}

Then, let us prove Lemma \ref{l.diag}.
\begin{proof}[Proof of Lemma \ref{l.diag}]
First, we fix $C_1>0$ and $\chi <0$. Then, 
apply Lemma $\ref{l.distor}$ for this $C_1$ letting 
$\kappa = \exp(-\chi/4)$ and fix $\tau$.
Let us see how we fix $\alpha$. 
We fix $\alpha$
sufficiently large so that the following holds:
if $B \in \mathrm{SL}(2, \mathbb{R})$ is the matrix 
representing the change of the basis with $\|P\|> \alpha$, 
then the corresponding angle is less than $\tau$.

For such choice of $\alpha$, we show the existence of 
good $i$ where the corresponding 
matrix $B_i$ has norm less than $\alpha$.
Suppose not, that is, every $B_i$ has norm greater than $\alpha$.
This implies that on each $T_{f^{i}(x)}|_{E^s}$, the image 
of eigenvectors have angle less than $\tau$. 
Then, Lemma \ref{l.distor} (using it inductively) 
implies that if $u$ and $v$ are 
eigenvectors in $V_0$, then for each $i$ we have
$\exp((\chi/4)i) < \|(A_{i-1}\cdots A_0)(v)  \|/\|(A_{i-1}\cdots A_0)(u)  \| 
< \exp(-(\chi/4)i)$,
but this contradicts to the hypotheses on the Lyapnov exponents
of the first return map $A_{n-1}\cdots A_0$.
\end{proof}

\subsubsection{Creating homothety}

Next lemma, inspired from \cite{BDP} or \cite{S} provides the announced homothety which is used to realize the division of 
perturbations:

\begin{lemm}\label{l.homoth}
Let $f\in \mathrm{Diff}^1(M)$ and $\La$ be a non-trivial basic set of 
stable index $2$.
Suppose there exists a hyperbolic periodic point $q$
which has a contracting complex eigenvalue.
Then,  there exists  diffeomorphisms $g$,  $C^1$-arbitrarily close to $f$
which has a periodic point $r$ such that 
Whose differential to $dg^{\mathrm{per}(r)}$ 
restricted to the stable direction is a contracting homothety.
Furthermore, the support of the perturbation 
from $f$ to $g$ can be taken 
arbitrarily close to an orbit of a periodic point of $f$.
\end{lemm}

The proof of Lemma \ref{l.homoth} is 
essentially done in \cite{BDP} (see Proposition 2.5 of \cite{BDP})
or \cite{S} (see Lemma 3.3 of \cite{S}). Thus we only present 
the sketch of the proof. 
\begin{proof}[Sketch of proof] 
The proof is a direct consequence
of Lemma~\ref{l.annihi} and a variant of Lemma ~\ref{l.transition1}: 
The idea of Lemma~\ref{l.transition1} is that there are periodic 
orbits whose differential restricted to the 
stable direction is an arbitrarily small 
perturbation of a product of a fixed transition matrix $T$ 
with an arbitrary power of the differential of the point $q$, 
which has a complex stable eigenvalue. 
Large powers of a matrix in $\mathrm{GL}(2,\RR)$ with 
a complex eigenvalue admits perturbations 
so that the product is an homothety. 
Therefore, $\La$ contains 
periodic orbits whose stable derivative are arbitrarily 
small perturbation of the product of the transition 
matrix $T$ by an arbitrary power of an homothety.  
Now Lemma~\ref{l.annihi} allows to perform a small 
perturbation along the orbit to cancel the
fixed intermediate  differentials.  
As a result, such periodic orbit has a homothety
to the stable direction. 
\end{proof}

\subsubsection{Creating flexible points: the proof of Proposition~\ref{p.weak-flexible}}

Now we are ready for starting the proof of 
Proposition~\ref{p.weak-flexible}. Let us start.
\begin{proof}
Let $C >1$ and $\chi<0$ be given, and
let $K$ be the hyperbolic periodic points of stable index $2$
containing periodic points $p$ and $q$ as is in the statement. 

According to Lemma~\ref{l.diag}, 
we can fix coordinates on $T\mathcal{O}(p)|_{E^s}$
so that we may assume that:
\begin{itemize}
 \item 
At the point $p$ of period $\pi(p)$
the first return map has the form of diagonal matrix
 \[
Df^{\pi(p)}|_{E^s}(p) =  P:=
\begin{pmatrix}
\lambda_1 & 0 \\
0              &  \lambda_2
\end{pmatrix}
\]
 with $|\lambda_1| \in[1-\nu,1)$  and the
smaller Lyapunov exponent 
$\left( 1/{\pi(p)} \right) \log|\lambda_2|$ is less that $\chi$.
\end{itemize}

Furthermore, according to Lemma~\ref{l.homoth}, by 
giving arbitrarily small perturbation
whose support is away from some neighborhood of 
$\mathcal{O}(p)$ and changing $r$ with $q$, 
we can assume that:
\begin{itemize}
\item At the point $q$, the first return map is a contracting homothety,
that is, we have
\[
Df^{\pi(q)}|_{E^s}(q) = Q :=
\begin{pmatrix}
r & 0 \\
0              &  r
\end{pmatrix},
\]
where  $\pi(q)$ denotes the period of $q$ and $r \in (0, 1)$. 
\end{itemize}

We choose $\nu$ so that multiplying $Df$ by the homothety of ratio $(1-\nu)^{-2}$ is a $\varepsilon/2$ perturbation.
Note that the choice of $\nu$ can be determined only 
from the value of $C$ and $\chi$.

For proving Proposition~\ref{p.weak-flexible}, it remains to 
show that arbitrarily small perturbation of $f$ may create $\varepsilon$-flexible point in $\La$.
We apply Lemma \ref{l.transition1}: there are transition matrices $T_1,T_2$ so that
 given any  $L = (l_1, l_2, l_3, l_4)$ with $(l_1,l_2)\neq (l_3,l_4)$,
we can find a periodic point 
$x_L$ whose first return map of differential cocycle 
restricted to the stable direction admits 
a small perturbation so that the first return map is:
\[  
T_2 Q^{l_4} T_1 P^{l_3}T_2 Q^{l_2} T_1 P^{l_1}.
\]
Indeed, by fixing $\varepsilon_2$ in the statement of 
Lemma \ref{l.transition1} (which can be taken 
arbitrarily small),
we can assume, by performing an 
$\varepsilon_2$-perturbation, that the differential 
along $x_L$ is indeed given by this matrix.
Now, let us choose $L$ so that the point $x_L$ can be 
perturbed to an $\varepsilon$-flexible point. 

 \subsubsection{First case:   
$T_i \in \mathrm{GL}_{+}(2, \mathbb{R})$, and $\lambda_i>0$}
From here the proof varies depending on 
the signature of $\det (T_i)$ and $\lambda_i$.
First, we consider the case
$T_1, T_2 \in \mathrm{GL}_{+}(2, \mathbb{R})$ 
and $\lambda_{1}, \lambda_2>0$.
The case where some of them are 
in $\mathrm{GL}_{-}(2, \mathbb{R})$  or $\lambda_i <0$
can be treated very similarly. We will explain 
how one can adapt the proof in the other 
cases at the end of this proof.

Let us continue the proof.
First, we apply Lemma \ref{l.annihi}
(letting $Q=Q$ and $T_1 = T$). 
Then we can find a sequence of matrices $(L_i)$ ($i = 0,\ldots, i_1-1$)
such that each $L_i$ is arbitrarily close to the homothety $Q$
and $T_2 (\prod_{i=0}^{i_1-1}L_i) = c_1 \mathrm{Id}$. 
Similarly, we take $(J_i)$ $(i = 0,\ldots, i_2-1)$ so that 
each $J_i$ are arbitrarily close to $Q$ and
$(\prod_{i=0}^{i_2-1} J_i)T_1 =  c_2 \mathrm{Id}$ holds.  
We also fix an integer $i_3$ such that
such that the multiplication
of the rotation $R(\pi t/2 i_3 )$ to the matrix $Q$ 
is an arbitrarily small perturbation of $Q$ for every $t \in [-1, 1]$. 

Then, let us fix $L = (l_1, l_2, l_3, l_4)$ as follows:
\begin{itemize}
\item $l_1 = l_3$ are sufficiently large integers 
so that $\mathcal{O}(x_L)$ contains $\mathcal{O}(\gamma_n)$
in its $\varepsilon/2$-neighborhood 
with respect to the Hausdorff distance (see Remark \ref{r.distance}).
\item $l_2 \neq l_4$ and $l_2, l_4 > i_1 + i_2 + i_3$.
\item The following inequality holds:
\[
\mu_L := \exp \left( \frac 1{\pi(x_L)}\cdot\log\left( (c_1c_2)^2 r^{l_2 + l_4 - 2(i_1 +i_2)}\lambda_1^{l_1+l_3}\right) \right) > 1 -\nu.\]
\end{itemize}
We can take such $L$ as follows: first take $l_2, l_4$ 
satisfying the second condition 
and later take $l_1$ and $l_3$ so large 
that the rest of the conditions are satisfied
(remember that 
$\pi(x_L) =(l_1 +l_3)\pi (p) + (l_2 +l_4) \pi (q) + 2(j_1 +j_2)$,
thus just by taking large $l_1, l_3$ we can obtain the inequality
in the third condition).

First, we perform a preliminary 
perturbation along $x_L$ so that the following holds
(we denote the perturbed map also by $f$):
\begin{itemize}
\item For $k =  K\pi (q)+ l_1\pi (p)+ j_1$ or 
$k= K\pi_2+ (l_1+l_3) \pi(p) + l_2\pi (q)+ 2j_1+j_2$
($K=0, \ldots, i_1 -1$), we have
\[\prod_{i=0}^{\pi (q) -1} 
Df(f^{k+i} (x_L) )= L_k.\]
\item  For $k =  K\pi (q)+ l_1\pi (p)+ j_1$ or 
$k = K\pi (q)+ (l_1+l_3) \pi (p) + l_2\pi (q)+ 2j_1+j_2$
($K=i_0, \ldots, i_1+i_2 -1$), we have
\[
\prod_{i=0}^{\pi (q) -1} 
Df(f^{i} (x) )= J_k.\]
\end{itemize}
At this moment, by using the fact 
$(\prod J_i) T_1 =c_2 \mathrm{Id}$, 
$T_2 (\prod L_i) = c_1 \mathrm{Id}$ 
and commutativity of homothetic transformation,
we can check that the first return map of the 
derivative of $f$ along $x_L$ has the 
following form: 
\begin{align}
(c_1 c_2)^2 Q^{l_4 -i_1-i_2}  P^{l_3} Q^{l_2 -i_1-i_2}  P^{l_1}
 =(c_1 c_2)^2Q^{l_2 + l_4 - 2(i_1 +i_2)}  P^{l_1+l_3}, \tag{$\dagger$}
\end{align}
in particular, 
$x_L$ has two real contracting eigenvalues 
$(c_1c_2)^2 r^{l_2 + l_4 - 2(i_1 +i_2)}\lambda_j^{l_1+l_3}$, $j=1,2$. 

We show that $x_L$ is $\varepsilon$-flexible for this $f$
by construction the path of linear cocycles for the flexibility. 
Recall that the path $\cA_t$  ($t\in[-1,1]$) we need to construct is the 
one which joins the cocycle $\cA_0$ induced by $Df$ on 
the stable bundle to a cocycle $\cA_{-1}$ whose return map is a homothety, and 
to a cocycle $\cA_{1}$ having an eigenvalue equal to $1$. 

First we build the path from $\cA_0$ to $\cA_1$.  
For that we only need to multiply $\cA_0$ along the orbit 
of $x_L$ by a homothety of ratio $\mu_L^{-1} \cdot \exp(t)$.
This gives a path of contracting cocycles
 between the original one and 
a homothetic one. 
Our second condition on the choice of $L$ implies 
that the ratio of this homothety is 
always less than $(1-\nu)^{-2}$ (and greater than one).
Our choice of $\nu$ implies that the multiplication by such 
a homothety remains an $\varepsilon$-perturbation of $\cA_0$.

Now let us build the path $\cA_t$ for $t\in[-1,0]$.  
For that purpose, 
we rewrite the product of the differential in the following way:
\begin{align*}
&T_2 \left(\prod_{k=0}^{i_1-1}L_k\right) Q^{l_4-i_1-i_2}
\left(\prod_{k=0}^{i_2-1}J_k \right)  T_1
P^{l_3}T_2 \left(\prod_{k=0}^{i_1-1}L_k\right) Q^{l_2-i_1-i-2} 
\left(\prod_{k=0}^{i_2-1}J_k \right) T_1   P^{l_1}\\
=&\left(c_1\mathrm{Id}\right) \underbrace{(Q^{i_3})}_{(\ast \ast )}  Q^{l_4-i_1-i_2-i_3} \left(c_2 \mathrm{Id}\right) 
P^{l_1} \left(c_1\mathrm{Id}\right)\underbrace{(Q^{i_3})}_{(\ast )} Q^{l_2-i_1-i_2-i_3} \left(c_2 \mathrm{Id}\right) P^{l_1}.
\end{align*}
In this product, 
we replace each of the homothetic matrix $Q$ in the first $Q^{i_3}$ 
(indicated by $(\ast)$) by  
$R(\frac{\pi t}{2 i_3 })\circ Q$ and 
the ones in the second one  
(indicated by $(\ast \ast)$) by 
 $R(\frac{- \pi t}{2 i_3} )\circ Q$.
By the choice of $i_3$, 
this perturbation can be done very small for every $t \in [-1, 0]$,
especially, it can be done with size less than $\varepsilon$. 
The effect of this perturbation on the product matrix is to replace
the whole product in ($\dagger$) to 
\[
(c_1 c_2)^2Q^{l_2 + l_4 - 2(i_1 +i_2)}  R\left(-{\frac{\pi}2}t\right)
P^{l_1} R\left( {\frac\pi 2}t \right)P^{l_3}.
\]
Now Lemma~\ref{l.homoper} ensures (remember $l_1 = l_3$) that the product matrix 
of $\cA_t$ in the period has two different
real contracting positive eignevalues for $t\neq -1$ and is a homothety for $t=-1$.  

Thus combining these two paths, we have completed the proof.
\end{proof}

\subsubsection{Other cases: matrices in $\mathrm{GL}_{-}(2,\RR)$
or negative eigenvalues}

Finally, let us consider the case where some of the signs are negative.

First, in the case where $\lambda_1$ or $\lambda_2$ are negative,  we just need to take $l_1, l_3 $ to be even numbers:
this replaces the matrix $P$ by $P^2$ everywhere in the proof, 
and the proof works identically.

In the case where one but only one of the transition 
$T_i$, say $T_1$ for instance, reverses the orientation. 
Then \cite{BDP} allows us to consider  the matrix 
$\tilde T_1= T_1P^i T_2 Q^j T_1$ 
as a new transition substituting $T_1$ 
(keeping $T_2$ unchanged as the other transition matrix), and now 
$ \tilde{T}_1,T_2$ 
both belongs to $\mathrm{GL}_+(2,\RR)$. 

It remains the case where $T_1$ and $T_2$ are both orientation reversing.  
In this case, we apply Lemma \ref{l.annihi} to the matrix
$$
T_i\begin{pmatrix}
1 & 0 \\
0              &  -1
\end{pmatrix},
$$
which provides a sequence of matrices $(L_i)$ ($i = 0,\ldots, i_1-1$)
arbitrarily close to the homothety $Q$
so that  
$$T_2 \left( \prod_{i=0}^{i_1-1}L_i \right) = (c_1 \mathrm{Id})\circ \begin{pmatrix}
1 & 0 \\
0              &  -1
\end{pmatrix} .$$ 
Similarly, we take $(J_i)$ $(i = 0,\ldots, i_2-1)$ arbitrarily close to $Q$ so that
$$
\left( \prod_{i=0}^{i_2-1} J_i \right)T_1 =  (c_2 \mathrm{Id})\circ
\begin{pmatrix}
1 & 0 \\
0              &  -1
\end{pmatrix}.
$$  
Then the proof works identically, just noticing that the matrix $\begin{pmatrix}
1 & 0 \\
0              &  -1
\end{pmatrix}$ is an involution which commutes with $P$ and $Q$.

\begin{rema}
In the proof, we can assume that the first return map of
differential to the unstable direction
is also orientation preserving
by adjusting the number of $l_i$ if necessary. 
\end{rema}

\vskip 5mm
\begin{tabular}{ll}
Christian Bonatti & Katsutoshi Shinohara
\\
\footnotesize{bonatti@u-bourgogne.fr} &  \footnotesize{herrsinon@07.alumni.u-tokyo.ac.jp}
\\
Institut de Math\'ematiques de Bourgogne, & CNRS - UMR 5584
\\
Universit\'e de Bourgogne 9 av. A. Savary,&
21000  Dijon, France

\end{tabular}

\end{document}